\documentclass[12pt]{amsart}
\usepackage{amsthm,amsmath,amssymb,amscd,graphics,enumerate, stmaryrd,xspace,verbatim, epic, eepic,color,url}
\usepackage{enumitem}
\usepackage{chngcntr}

 \usepackage[mathscr]{eucal}
\usepackage{hyperref}

\usepackage[all]{xypic}
\SelectTips{cm}{}

\newtheorem{thm}{Theorem}[subsection]
\newtheorem{theorem}[thm]{Theorem}
\newtheorem{defn}[thm]{Definition}
\newtheorem{prop}[thm]{Proposition}
\newtheorem{proposition}[thm]{Proposition}
\newtheorem{lemma}[thm]{Lemma}
\newtheorem{corollary}[thm]{Corollary}
\theoremstyle{remark}
\newtheorem{remark}[thm]{Remark}
\newtheorem{conv}[thm]{Convention}
\newtheorem*{construction}{Construction}

\theoremstyle{definition}
\newtheorem{definition}[thm]{Definition}
\newtheorem{assumption}[thm]{Assumption}
\newtheorem{example}[thm]{Example}

\newcommand{\marg}[1]{}
\newcommand{\Barbara}[1]{\marg{Barbara: #1}}
\newcommand{\Dan}[1]{\marg{Dan: #1}}
\newcommand{\Jonathan}[1]{\marg{Jonathan: #1}}
\newcommand{\Chuck}[1]{\marg{Chuck: #1}}

\newcommand{\cA}{{\mathcal{A}}}
\newcommand{\cB}{{\mathcal{B}}}
\newcommand{\cE}{{\mathcal{E}}}
\newcommand{\eps}{\epsilon}
\newcommand{\cO}{{\mathcal{O}}}
\renewcommand{\cL}{{\mathcal{L}}}
\newcommand{\cM}{{\mathcal{M}}}
\newcommand{\cU}{{\mathcal{U}}}

\newcommand{\cT}{{\mathcal{T}}}
\newcommand{\thickslash}{\fatslash\,}
\newcommand{\cX}{{\mathcal{X}}}
\newcommand{\Id}{{\mathrm{id}}}
\renewcommand{\cD}{{\mathcal{D}}}
\newcommand{\cW}{{\mathcal{W}}}
\newcommand{\smooth}{{\mathrm{smooth}}}
\newcommand{\ocM}{{\overline{\mathcal{M}}}}

\newcommand{\bC}{\mathbf{C}}

\newcommand{\bN}{\mathbf{N}}

\newcommand{\bZ}{\mathbf{Z}}

\newcommand{\bc}{\mathbf{c}}
\newcommand{\br}{\mathbf{r}}
\newcommand{\bs}{\mathbf{s}}

\newcommand{\bDelta}{{\boldsymbol\Delta}}
\newcommand{\bmu}{{\boldsymbol\mu}}

\newcommand{\fC}{\mathfrak{C}}
\newcommand{\fM}{{\mathfrak{M}}}
\newcommand{\fN}{{\mathfrak{N}}}

\newcommand{\fT}{{\mathfrak{T}}}

\newcommand{\fm}{\mathfrak{m}}

\newcommand{\fr}{\mathfrak{r}}

\newcommand{\hA}{\widehat{A}}
\newcommand{\hL}{\widehat{L}}
\newcommand{\hS}{\widehat{S}}
\newcommand{\hs}{\hat{s}}
\newcommand{\hT}{\widehat{T}}

\newcommand{\oalpha}{\overline{\alpha}}
\newcommand{\opi}{\overline{\pi}}
\newcommand{\oA}{\overline{A}}
\newcommand{\oC}{\overline{C}}
\newcommand{\oE}{\overline{E}}
\newcommand{\oM}{\overline{M}}

\newcommand{\om}{\overline{m}}
\newcommand{\ox}{\overline{x}}
\newcommand{\oX}{\overline{X}}

\newcommand{\osX}{\overline{\mathscr{X}}}
\newcommand{\osY}{\overline{\mathscr{Y}}}

\newcommand{\sA}{\scr{A}}
\newcommand{\sB}{\scr{B}}
\newcommand{\sC}{\mathscr{C}}
\newcommand{\sD}{\scr{D}}
\newcommand{\sE}{\mathscr{E}}
\newcommand{\sL}{\scr{L}}
\newcommand{\sP}{\scr{P}}
\newcommand{\sS}{\scr{S}}
\newcommand{\sT}{{\mathscr{T}}}

\newcommand{\sX}{\scr{X}}
\newcommand{\sY}{\scr{Y}}
\newcommand{\sZ}{\mathscr{Z}}

\newcommand{\tS}{\widetilde{S}}
\newcommand{\ts}{\widetilde{s}}

\newcommand{\tY}{\widetilde{Y}}

\newcommand{\toset}{\rm to}

\newcommand{\BGm}{\cB\mathbb{G}_m}

\newcommand{\uHom}{\underline{\Hom}}

\usepackage{ifthen}

\newcommand{\setlike}[4]{
  { \left#3 
      \ifthenelse{\not\equal{#1}{}}
      {\left. #1 \vphantom{#2} \: \right| {#2} \:}
      {#2}
    \right#4
  }
}

\providecommand{\set}[2][]{{\setlike{#1}{#2}{\{}{\}}}}

\RequirePackage{ifthen}

\providecommand{\mathscr}{\mathcal}

\def\scr#1{\mathscr{#1}}

\def\cal#1{\mathcal{#1}}

\newcommand{\tensorlike}[2]  
{
  \ifthenelse   
  {
    \equal{#1}{}
  }
  {#2}          
  {		
    \underset{#1}{#2}
  }
}

\providecommand{\dblar}[1]{\ar@<2pt>[#1] \ar@<-2pt>[#1]}

\newcommand{\Ext}{\operatorname{Ext}}
\newcommand{\cExt}{{\mathcal{E}xt}}

\newcommand{\Gm}{\mathbb{G}_m} 

\newcommand{\Hom}{\operatorname{Hom}}

\newcommand{\Isom}{\operatorname{Isom}}
\newcommand{\uIsom}{\underline{\Isom}}
\newcommand{\Li}{{\operatorname{Li}}}
\newcommand{\Morph}[1][{}]{\ifthenelse{\equal{{#1}}{{}}}{\operatorname{Morph}}{\operatorname{{#1}-Morph}}}

\newcommand{\Spec}{\operatorname{Spec}}

\renewcommand{\AA}{{\mathbb A}}
\newcommand{\CC}{{\mathbb C}}
\newcommand{\LL}{{\mathbb L}}

\newcommand{\PP}{{\mathbb P}}

\newcommand{\RR}{{\mathbb R}}
\newcommand{\ZZ}{{\mathbb Z}}
\newcommand{\GG}{{\mathbb G}}

\newcommand{\id}{\operatorname{id}}

\newcommand{\naive}{\text{na\"\i ve}}

\renewcommand{\bar}[1]{\overline{#1}}
\newcommand{\colim}{\varinjlim}
\newcommand{\rest}[1]{\big|_{#1}}

\newcommand{\tensor}{\mathop{\otimes}}
\newcommand{\fp}{\mathop{\times}}
\newcommand{\cp}{\mathop{\amalg}}
\newcommand{\fprod}[1][]{\fp_{#1}}

\newcommand{\cat}[1]{\mathsf{#1}}

\newcommand{\Arel}{\cA^{\mathrm{rel}}}
\newcommand{\Zrel}{\mathfrak{Z}^{\mathrm{rel}}}

\newcommand{\twst}{{\rm tw}}
\newcommand{\tw}{{\rm tw}}
\newcommand{\sst}{{\rm ss}}
\newcommand{\spl}{{\rm spl}}
\newcommand{\twss}{{\rm twss}}
\newcommand{\maps}{{\rm maps}}

\newcommand{\cTtw}{\cT^\tw}
\newcommand{\fTtw}{\fT^\tw}
\newcommand{\cTLi}{\cT_{\rm Li}}
\newcommand{\fTLi}{\fT_{\rm Li}}

\newcommand{\cTnaive}{\cT_\naive}
\newcommand{\fTnaive}{\fT_\naive}
\newcommand{\cTnaivetw}{\cT^\tw_\naive}
\newcommand{\fTnaivetw}{\fT^\tw_\naive}

\newcommand{\cTmaps}{\cT_\maps}

\newcommand{\cTmapstw}{\cT^\tw_\maps}

\newcommand{\Tpairs}{\cT}
\newcommand{\Tmaps}{\cT_{\rm maps}}
\newcommand{\Tdeg}{\fT}
\newcommand{\TGV}{\cT_{\rm GV}}
\newcommand{\Tsc}{\cT_{\rm sc}}
\newcommand{\Ttoset}{\cT_{\toset}}
\newcommand{\fTtoset}{\fT_{\toset}}
\newcommand{\TLi}{\cT_{\rm Li}}
\newcommand{\Tlog}{\cT_{\log}}

\newcommand{\Tmid}{\widetilde{\cT}}
\newcommand{\Tmapsdeg}{\cT_{\rm maps,deg}}

\newcommand{\Tspl}{\cT_{\rm spl}}
\newcommand{\Ttwspl}{\cT_{\spl}^{\twst}}
\newcommand{\Ttwpairs}{\Tpairs^{\twst}}
\newcommand{\Ttwdeg}{\Tdeg^{\twst}}
\newcommand{\TtwGV}{\TGV^{\twst}}

\newcommand{\Xlog}{\sX_{\log}}

\providecommand{\fMtw}{\ensuremath{\fM^{\tw}}}

\title{Expanded degenerations and pairs}
\author[D. Abramovich]{Dan Abramovich}
\thanks{Research of D.A. partially supported by NSF grants DMS-0603284 and DMS-0901278}
\address[Abramovich]{Department of Mathematics, Box 1917, Brown University,
Providence, RI, 02912, U.S.A}
\email{abrmovic@math.brown.edu}
\author[C. Cadman]{Charles Cadman}
\address[Cadman]{Atlanta, GA}
\email{math@charlescadman.com}
\author[B. Fantechi]{Barbara Fantechi}
\thanks{Research of B.F. partially supported by: GNSAGA, prin, Dubrovin}
\address[Fantechi]{SISSA,
Via Beirut 4,
34014 Trieste, Italy}
\email{fantechi@sissa.it}
\author[J. Wise]{Jonathan Wise}
\address[Wise]{Stanford University, Department of Mathematics, building 380, Stanford, California 94305}
\email{jonathan@math.stanford.edu}
\thanks{Research of J.W. partially supported by NSF-MSPRF 0802951}

\date{\today}

\begin{document}

\maketitle

\begin{abstract}
Since Jun Li's original definition, several other definitions of expanded pairs and expanded degenerations have appeared in the literature.  We explain how these definitions are related and introduce several new variants and perspectives.  Among these are the twisted expansions used by Abramovich and Fantechi as a basis for orbifold techniques in degeneation formulas.
\end{abstract}

\setcounter{tocdepth}{1}
\tableofcontents

\section{Introduction}

\subsection{Expansions in Gromov--Witten theory}
Let  $X \rightarrow B$ be a flat morphism from a smooth variety to a
smooth curve, with a unique critical value $b_0\in B$.  Suppose that the critical fiber $X_0$ is the  union of smooth varieties $Y_1$ 
and $Y_2$ along a closed subscheme $D$ that is a smooth divisor in each $Y_i$.  Working over $\CC$ and assuming that $X$ is projective over $B$, Jun Li associated a family of \emph{expanded degenerations} $X_{\Li}^{\exp} \to \fTLi(X/B,b_0)$ to $X/B$ (the symbol $\fT$ is the Gothic letter {\it T}).  The fibers of $X_{\Li}^{\exp}$ over $\fTLi(X/B,b_0)$ are different semistable models of $X_0$.  
A typical degenerate fiber of $X_{\Li}^{\exp}$ over $\fTLi(X/B,b_0)$ consists of a chain
\begin{equation} \label{Eq:deg}
 Y_1 \mathop{\sqcup}\limits_D P \mathop{\sqcup}\limits_D \cdots   \mathop{\sqcup}\limits_D  P\mathop{\sqcup}\limits_D
  Y_2
\end{equation}
where $P:= \PP_D(\cO\oplus N_{D/ Y_1})$ is a $\PP^1$ bundle over $D$ (see Definition~\ref{Def:standard-deg}). The copies of $D$ used in this gluing are called {\em splitting divisors}. 
 
Li also defined related families of \emph{expanded pairs} (see Definition~\ref{Def:standard-pair}), denoted here by $(Y^{\exp}_i,D^{\exp})_{\Li} \to \cTLi(Y_i,D)$.  These parameterize different models for the pair
$(Y_i, D)$, obtained by splitting the chain above along a splitting divisor:
\begin{equation} \label{Eq:pair}
 Y_1 \mathop{\sqcup} P \mathop{\sqcup} \cdots   \mathop{\sqcup}  P.
\end{equation}

We record Li's definitions of $\fTLi(X/B, b_0)$ and $\cTLi(Y_i,D)$  in Definition \ref{def:Li-deg} and Section \ref{sec:exp-pairs-li}. It has been observed that the base stacks $\fTLi(X/B, b_0)$ and $\cTLi(Y_i,D)$ should be independent of the varieties $X$ and $Y_i$ involved, but to our knowledge a proof has not yet appeared. This will follow from Theorem \ref{Th:main-T} below.\Dan{Maybe insert the statement of theorems here?}

These stacks form the foundation for Li's approach to the {\em
  degeneration formula}.   This formula originated in symplectic
geometry in the work of A.-M.\ Li and Y.\ Ruan \cite{Li-Ruan} (see also
\cite{Ionel-Parker}) relating Gromov--Witten invariants of the generic
fiber of  $X \rightarrow V$ to appropriately defined  {\em relative} Gromov--Witten invariants of the pairs $(Y_i, D)$.

The stacks $\fTLi(X/B, b_0)$ and $\cTLi(Y,D)$ continue to be important in
Gromov--Witten theory, with new applications and generalizations
emerging in the orbifold theory \cite{AF,ACW} and in related theories, e.g.
 \cite[Sec 3.3]{Tzeng}.

 Li's definitions are geometrically appealing, but are necessarily subtle for reasons explained below. We are not aware of an earlier treatment showing that  $\fTLi(X/B, b_0)$ and $\cTLi(Y,D)$ are algebraic stacks, or detailing the relationships between them necessary for the degeneration formula. Such statements are proven in this paper. 
 
\subsection{The definition of Graber and Vakil}
Anticipating that the stacks of expansions are independent of the target, Graber and Vakil constructed them as the moduli stacks of expansions of curves.  We record their definition in Definition~\ref{Def:TGV} and call their moduli space $\TGV$. One appealing feature of Graber's and Vakil's definition is its close relationship to the already well-understood moduli space of Deligne--Mumford pre-stable curves:  the algebraicity of $\TGV$ is immediate, for example.

Graber and Vakil's define $\TGV$ so that it clearly  parametrizes expansions of the pair $(\PP^1,0)$; at least intuitively, one can construct an expansion of any pair $(Y,D)$ from an expansion of $(\PP^1,0)$ by ``doing the same thing'' to $(Y,D)$ as was done to $(\PP^1,0)$.  However, we do not know of anywhere that this procedure is spelled out, nor of a demonstration of its equivalence to Li's definition.  This equivalence follows from our Theorem \ref{Th:main-T}.

A central  impetus for this paper is the need to generalize expansions to {\em twisted expansions}, as applied in \cite{AF,ACW}.  Taking the point of view of Graber and Vakil, twisted expansions replace the semistable curves appearing in $\TGV$ with \emph{twisted semistable curves} \cite{AV, AOV}.  Our approach to the stacks of expansions $\cT$ and $\fT$ affords an immediate generalization to stacks $\cT^\tw$ and $\fT^\tw$ of twisted expansions (see Definition \ref{Def:twisted-exp}).

\subsection{Results: algebraicity and comparison}

In Definitions  \ref{Def:pair-exp} and \ref{Def:deg-exp} we introduce stacks $\fT$ and $\cT$ of expanded degenerations and expanded pairs, along with their universal families, with no restrictions on $X$ or $Y$.  For pairs, our definition of $\cT$ is similar to Graber's and Vakil's, but the pair $(\PP^1, 0)$ is replaced by the \emph{universal pair} $(\sA, \sD)$, where $\sA = [\AA^1/\GG_m]$ and $\sD = [0\,/\,\GG_m]$.  The universal virtues of $\sA$ are explained in \cite[Lemma 2.1.1]{cadman}: any pair $(Y,D)$ admits a canonical morphism to $(\sA,\sD)$ such that $D = Y \fp_{\sA} \sD$.  An expansion $(\sA',\sD')$ of $(\sA,\sD)$ induces an expansion of $(Y,D)$ by base change:  $(Y',D')= (Y\times_\sA \sA', D\times_\sD \sD')$. 

For degenerations, we also use base change to reduce to the universal example.  The universal degeneration is the multiplication map $\sA^2 \rightarrow \sA$, in the sense that every degeneration $(X/B, b_0)$ admits a canonical map to $(\sA^2 / \sA, \sD)$ (see Section~\ref{Sec:deg-deg}).  Expansions of $(X/B,b_0)$ are defined by pulling back expansions of $(\sA^2 / \sA, \sD)$.

To prove that $\fT$ and $\cT$ are algebraic we provide an isomorphism with stacks of curves similar to $\TGV$.

We prove the following:

\begin{theorem}[Algebraicity theorem]\label{Th:twisted-algebraic}
The stacks $\fTtw$ and $\cTtw$ are algebraic stacks, locally of finite type over $\ZZ$, containing the stacks $\fT$ and $\cT$ as open and dense substacks.
\end{theorem} 

\begin{theorem}[Comparison theorem]\label{Th:main-T}
\begin{enumerate} 
\item We have an isomorphism $\cT \simeq \TGV$.
\item Assume $Y$  is smooth and projective over a field, and $D$ is irreducible. Then we have a canonical  isomorphism $\cTLi(Y,D) \simeq \cT$. 
\item Assume $X$ above  is projective over $B$ and $D$ is connected. Then we have a canonical isomorphism $\fTLi(X/B, b_0)\simeq \fT\times_\sA B$.
\end{enumerate}
\end{theorem}

\subsection{Split expansions and gluing maps}\Dan{to be rewritten}
A key element of the degeneration formula, which too easily might slip from one's attention, is a {\em gluing map} relating the stacks $\fT$ and   $\cT$ as well as their twisted versions. Let us first describe the untwisted situation.

Define $\fT_\spl= \cT\times \cT$. We can view it as parametrizing pairs consisting of an expansion of $(Y_1,D)$ together with an expansion of $(Y_2,D)$. Gluing these along the isomorphism 
\begin{equation} \label{Eq:normal-isom}
N_{D/ Y_1} \otimes N_{D/ Y_2} \simeq O_D
\end{equation}
we get an expansion of $X_{0}$ together with a choice of a splitting divisor. In fact $\fT_\spl$ is canonically the stack parametrizing expansions of $X_0$ together with  a choice of a splitting divisor. Forgetting this choice gives an object of  $\fT_0 = 0\times_\sA\fT$. We obtain the following:

\begin{proposition}\label{Prop:gluing}
There is a canonical morphism $\fT_\spl \to \fT_0$ of pure degree 1.
\end{proposition}

This morphism is implicit in Li's treatment of the degeneration formula; see \cite[Propositions~4.12 and~4.13]{Li1}.

The situation is slightly more subtle in the twisted case. First, in order to glue twisted expanded pairs the twisting along the marked divisor must coincide. Denoting by $\fr:\cTtw \to \ZZ_{>0}$ the locally constant twisting function, we consider the stack  $\cTtw\times_{\ZZ_{>0}} \cTtw$ of pairs of twisted expansions with the same twisting at the marking. Second, we do not have a morphism from this stack to $\fT$: the isomorphism (\ref{Eq:normal-isom}) must be lifted to the $\fr$-th roots. Denoting by $\fTtw_\spl$ the stack parametrizing such liftings, we obtain a pair of morphisms $\fTtw_\spl \to \cTtw\times_{\ZZ_{>0}} \cTtw$ and $\fTtw_\spl \to \fTtw_0 := 0\times_\sA \fTtw$. The analogous gluing result is the following proposition.  
 \begin{proposition}\label{Prop:gluing-tw}
 \begin{enumerate}
 \item The morphism  $\fTtw_\spl \to \cTtw\times_{\ZZ_{>0}} \cTtw$ is a gerbe banded by $\bmu_\fr$, and thus has degree $1/\fr$ 
 \item The morphism $\fTtw_\spl \to \fTtw_0$ has degree $1/\fr$. 
 \end{enumerate}
 \end{proposition}

This is used to prove the orbifold degeneration formula of \cite{AF}. The second morphism is not a gerbe but rather more like the embedding of a scheme in an $(r{-}1)$-st infinitesimal neighborhood.

\subsection{Labels and twists}\Dan{to be rewritten}
Another result of \cite{AF} shows that relative orbifold Gromov--Witten invariants are independent of twisting. This relies on Costello's formalism of labeling by a set $\sS$ \cite{Costello} applied to expanded degenerations.  For this purpose we define algebraic stacks $\cT^\sS$ and $\fT^\sS$ of expansions with splitting divisors labeled by $\sS$ (Section \ref{sec:label-by-set}) and show that they are algebraic and representable over the respective $\cT$ and $\fT$ (Proposition \ref{Prop:labels-algebraic}). Furthermore given a function $\fr: \sS \to \ZZ_{>0}$ we define algebraic stacks $\cT^\fr$ and $\fT^\fr$ of $\fr$-twisted $\sS$-labelled expansions, where the twisting of a splitting divisor labelled by $\delta\in \sS$ is   determined as $\fr(\delta)$; accordingly the function $\fr$ is called a {\em twisting choice}. We prove

\begin{proposition}\label{Prop:untwisting}
If $\fr, \fr'$ are twisting choices and $\fr$ divides $\fr'$, then the structure map $\cT^{\fr'} \to \cT^\sS$ factors canonically as $\cT^{\fr'} \to \cT^\fr \to\cT^\sS$. The same holds for $\fT^{\fr'} \to \fT^\fr \to\fT^\sS$.  If further $\fr'$ divides $\fr''$ then the resulting triangle is canonically commutative. \Dan{State the exact comparison needed in \cite{AF}?} \end{proposition}

The map $\cT^{\fr'} \to \cT^\fr$ is named the {\em partial untwisting map}.

\subsection{Na\"{\i}ve expansions} \Dan{to be rewritten}
The definitions employed here may not be the first one would consider.  Why not simply define a family of expansions to be an arbitrary flat family with fibers of the forms~\eqref{Eq:deg} or~\eqref{Eq:pair}?  We refer to these families as \emph{na\"ive expansions} and write $\cTnaive(Y,D)$ for pairs and $\fTnaive(X/B,b_0)$ for degenerations.  The problem with these definitions is that in any reasonable generality the automorphism group and deformation space of a na\"ive expansion are inappropriate; it is false, for example, that na\"ive expansions are independent of the target.  See Section~\ref{Sec:trouble} for more about the trouble with the na\"ive definition.




\subsection{Expansions and combinatorial structures} 
Jun Li's treatment also requires a number of  properness and connectedness assumptions. Under these assumptions one can also use logarithmic structures on expanded degenerations. We do not follow this approach in this paper, as it is studied in some detail elsewhere \cite{Kim, AMW}.\Jonathan{I don't understand this discussion.  How does one use logarithmic structures:  with expansions or without?} However, in Section~\ref{Sec:otherdef} we reinterpret our definitions of $\cT$ and  $\fT$ in more combinatorial terms, using logarithmic structures and configurations of line bundles.

\subsection{Acknowledgement}
We are happy to acknowledge Michael Thaddeus for several useful
conversations about the topics of Section~\ref{sec:TLi}
and Appendix~\ref{sec:Gm-action}.

\section{Definitions of the stacks}
\label{sec:defs}

\subsection{Expanded pairs: na\"{\i}ve and universal approaches}

Let $Y$ be either a  scheme or an Artin stack over a
field $k$, and
$D\subset Y$ a 
Cartier divisor. We call such $(Y,D)$ a {\em pair}.


Fix a pair $(Y,D)$ and write $\cO_Y(D)|_D = \cO_D(D)$. Let $P$ be the projective completion $\PP(\cO_D(D) \oplus \cO)$ of the
normal bundle of $D$.  It is a $\PP^1$-bundle over $D$ with two
sections, denoted $D_+$ and $D_-$ with normal bundles 
$N_{D_-/P} \simeq \cO_D(-D)$ and $N_{D_+/P} \simeq \cO_D(D)$.\Jonathan{notation conflict for normal bundles; resolve}  These isomorphisms  are \emph{canonical}
since the  complement of $D_+$ (resp. $D_-$) in $P$ is canonically
isomorphic to the total space of the line bundle $\cO_D(-D)$ (resp. $\cO_D(D)$) and $D_-$
(resp. $D_+$) is the image of the zero-section. 

The action of $\Gm(D)$ on the line bundle $\cO_D(D)$ induces an action of $\Gm(D)$ on $P$ whose fixed loci are $D_-$ and $D_+$.

\begin{definition}\label{Def:standard-pair}
The \emph{standard expansion of length $\ell$} of $(Y,D)$ is obtained by gluing $\ell$ copies of $P$ to $Y$ along the distinguished divisors:
\begin{equation*}
  Y[\ell]  :=  Y \mathop{\sqcup}\limits_{D=D_-}  P \mathop{\sqcup}\limits_{D_+=D_-}  \cdots   \mathop{\sqcup}\limits_{D_+=D_-}  P.
\end{equation*}

The notation is meant to indicate that the first copy of $P$ is joined
to $Y$ along $D_-$ and to the second copy of $P$ along $D_+$, etc.
Thus the normal bundles of each  copy of $D$ in the two
components containing it are dual to each other.  The divisor $D_+$ of the last copy
of $P$ is contained in the smooth locus of $Y[\ell] $.  Denoting this
divisor $D[\ell] $, there is a morphism of pairs 
\begin{equation*}
  (Y[\ell] , D[\ell] ) \rightarrow (Y, D) 
\end{equation*}
collapsing all of the copies of $P$ onto $D$.  This maps $D[\ell]$
isomorphically onto $D$.

 We
allow $\ell=0$ by setting $(Y[0] , D[0] ) := (Y, D)$. 
\end{definition}

Such expansions have appeared also under the name {\em accordions} \cite{GV}. A closely related notion is  {\em
  Zollst\"ocke} \cite{Kausz}.
We note that even when
$Y$ is a stack, the morphism $Y[\ell] \to Y$ is representable, and in
fact projective.

Here is the na\"\i ve notion of expanded pairs over a base scheme:

\begin{definition}  \label{Def:pair-naive}
\begin{enumerate}[leftmargin=0pt,itemindent=3em]
\item 
  Let $(Y, D)$ be a pair. A
  {\em na\"\i ve expansion of $(Y, D)$ over a scheme $S$} is a flat family of pairs $(Y',D')
  \to S$, locally of finite presentation over $S$,   with a proper map
  $(Y',D') \to (Y,D)$ such that each geometric fiber of $(Y',D')$ over $S$ is
  isomorphic over  $(Y, D)$  to a standard expansion of some length $\ell\geq 0$.

\item
A morphism from a na\"\i ve expansion $(Y',D')/S'$ to $(Y'',D'')/S''$ is a fiber diagram
$$\xymatrix{
Y' \ar[r]\ar[d] & Y''\ar[d] \\ S'\ar[r] & S'' .
}
$$
\end{enumerate}
\end{definition}

We note that  when $Y$ is a stack,  this construction a priori gives a 2-category, whose objects are  na\"\i ve expansions $(Y',D')/S$ 
of $(Y,D)$,  arrows given as above, and 2-arrows are isomorphisms of arrows. But Since the morphism $Y' \to Y$ is representable, this 2-category is  isomorphic to a category,  since 2-isomorphisms are unique when they exist; see \cite[Lemma 3.3.3]{AGV}. In fact it suffices that $Y'\to Y$ restricts to a representable morphism on a schematically dense substack $Y'_0\subset Y'$, see \cite[Lemma 4.2.3]{AV};  this is used when discussing twisted expansions (Definition \ref{Def:twisted-exp}).

 The functor $(Y',D')/S \mapsto S$ makes this a fibered category.

\begin{definition} Denote the resulting category $\cTnaive(Y,D)$.  Let $(Y^{\exp}, D^{\exp})_\naive\to \cTnaive(Y,D)$\Jonathan{I don't like the notation for the universal na\"ive expanded pair.  Do we actually use this object anywhere?  Maybe the notation $(Y^{\exp}_{\naive}, D^{\exp}_{\naive})$ would be better} be the universal na\"\i ve expansion. 
\end{definition}

We note that even when $Y$ is a scheme, an expansion $Y'$ is only guaranteed to be an algebraic space. An example that is not a scheme is given in Remark \ref{Rem:Li-etale}.

 We have the following:

\begin{lemma}\label{Lem:T-stack}
The category  $\cTnaive(Y,D)$ is a stack.
\end{lemma}

\begin{proof} Since the pullback of an expansion is an expansion, and since arrows are fiber diagrams, the category $\cTnaive(Y,D)$ is
fibered in groupoids over the category of schemes. The
descent properties are automatic since we allow a na\"\i ve expansion $Y'$ to be an algebraic
space or a stack.
\end{proof}

Definition \ref {Def:pair-naive} above is problematic for a number of reasons we detail in Section \ref{Sec:trouble}. But we will use it in the following particular case:

As discussed in the introduction, there is a remarkable stack $\sA  :=
[\AA^1 / \Gm]$ with a Cartier divisor $
\sD:= [0 / \Gm] \simeq \cB\GG_m$, which together form the {\em universal pair $(\sA,\sD)$}: given a pair $(Y,D)$,  the sheaf homomorphism
$\cO_X(-D) \rightarrow \cO$ induces a  morphism $Y \rightarrow \sA$ and $D$ is the pre-image of $\sD$. The morphism $Y \rightarrow \sA$ is flat since $D$ is a Cartier divisor.
\begin{definition}\label{Def:pair-exp}
Denote by $\cT:=\cTnaive(\sA,\sD)$ the category of na\"ive expansions of the pair $(\sA, \sD)$. Its universal family is accordingly denoted $(\sA^{\exp}, \sD^{\exp})\to \cT$.
\end{definition}

We can use this to redefine expansions for any pair:
\begin{definition}\label{Def:pair-exp-Y}
Let $(Y, D)$ be a pair with associated flat morphism $Y \to \sA$.  An {\em expansion of $(Y,D)$ over $S$} is a commutative diagram with fiber square
$$
\xymatrix{
(Y',D') \ar[r]\ar[d] & (\sA',\sD')\ar[d]\\
(Y,D)_S \ar[r]\ar[dr] & (\sA,\sD)_S\ar[d]\\
& S,}
$$
where the right column is a na\"\i ve expansion of $(\sA,\sD)$, and the arrow $(Y,D)_S \to (\sA,\sD)_S$ in the middle is the canonical arrow. 
\end{definition}

\begin{remark} We could define a stack $\cT(Y,D)$ parametrizing expansions of $(Y,D)$ but the universal property of fibered products gives a morphism $\cT \to \cT(Y,D)$, sending $(\sA',\sD')\to \sA$ to the diagram above,  which is evidently an isomorphism. This means that the stack $\cT$ is the stack of expansions of any pair $(Y,D)$.
Note that this is defined without properness or connectedness assumptions, and readily applies to any geometry where $\sA$ is the moduli stack of line bundles with sections. Part (1) of Theorem \ref{Th:main-T} says in particular that $\cT$ is an  {\em algebraic} stack.

Having the universal case $Y=\sA$ is one  reason to allow $Y$
to be a stack rather than an algebraic space. 
\end{remark}

Note that the arrow $(Y',D') \to (Y,D)_S$ on the left of the diagram above is a na\"\i ve expansion. This follows since na\"\i ve expansions satisfy a basic compatibility for flat morphisms:\Dan{this needs to be moved}
\begin{lemma}\label{Lem:exp-basic-comp}
Let $f:Z \to Y$ be a flat morphism and $(Y',D') \to (Y,D)$ a na\"\i ve 
expansion. Let $Z' = Z\times_YY'$, $E =Z\times_YD$   and
$E=Z\times_YD'$. Then $(Z',E') \to (Z,E)$ is a na\"\i ve expansion.
\end{lemma}

\begin{proof}
This follows since $f^*O_Y(D) = O_Z(E)$.
\end{proof}

\subsection{Types of degenerations}

Before considering expanded degenerations, we consider which types of degenerations we want to study.
\subsubsection{Basic degenerations}\label{Sec:basic-deg}
The most basic case of degeneration considered here is a flat morphism
$p : X \rightarrow B$ from a smooth variety to a smooth curve, with
unique critical value $b_0\in B$, where the fiber is the union of two
smooth varieties, $Y_1$ and $Y_2$, meeting along a smooth subvariety $D$ which
is itself a divisor in $Y_1$ and $Y_2$. Note that we have an
isomorphism of line bundles $\cO_X(Y_1) \otimes \cO_X(Y_2) \simeq
p^*\cO_B(b_0)$ such that the defining sections are compatible: $\mathbf{1}_{Y_1}
\tensor \mathbf{1}_{Y_2} = p^* \mathbf{1}_{b_0}$. 

\subsubsection{Degenerate fiber}\label{Sec:deg-deg} Along with the above case one might
consier the degenerate fiber abstractly: we have a scheme $X =
Y_1\sqcup_D Y_2$, with a specified isomorphism $\cO_{Y_1}(D)|_D \otimes
\cO_{Y_2}(D)|_D \simeq \cO_D$. 

\subsubsection{Letting the degeneration vary}\label{Sec:general-deg} Still more generally, we
can remove the assmption that $B$ is a curve or $b_0$ is a point: assume
$B$ is a scheme or algebraic stack, $\cL$ an invertible sheaf and $s:
\cL \to \cO_V$ a 
sheaf homomorphism, with associated closed subscheme $B_0 = \Spec
\cO_B/s(\cL)$. Consider a flat surjective morphism $p:X \to B$, two invertible 
sheaf homomorphisms $s_1:\cL_1\to \cO_X$ and $s_2:\cL_2\to \cO_X$ with
associated closed subschemes $Y_1$ and $Y_2$ intersecting along a subscheme
$D$, and an isomorphsim $\cL_1\otimes \cL_2 
\xrightarrow{\sim} p^* \cL$ carrying $s_1\otimes s_2$ to $p^*s$. We impose the
following non-degeneracy assumption: $D\subset Y_i$ should be a
Cartier divisor. It is clear that cases \ref{Sec:basic-deg} and
\ref{Sec:deg-deg} are special cases of \ref{Sec:general-deg}. 

\subsubsection{The universal case}\label{Sec:universal-deg}  A special case of
\ref{Sec:general-deg} is the following situation, which is the reason
why we want to allow $X$ to be a stack: let $X = \sA^2$, $B= \sA$ and
$p: X \to B$ induced by the multiplication map $\AA^2 \to \AA^1$
sending $(x_1,x_2)$ to $x_1x_2$. Let $Y_1 = \sD_1\times \sA$, $Y_2 =
\sD_2\times \sA$ and $B_0= \sD$. 

Clearly this satisfies the
non-degeneracy assumption in \ref{Sec:general-deg}. On the other hand,
given a situation as in \ref{Sec:general-deg}, the data $(\cL_i, s_i)$
and $(\cL_i, s_i)$ provide a 
canonical commutative diagram 
\begin{equation}\label{Eq:universal-deg}\xymatrix{
X \ar[r]\ar[d]& \sA^2\ar[d]\\ B \ar[r]& \sA.
}
\end{equation}
Nondegeneracy amounts to flatness of the map $X \to B\times_\sA
\sA^2$. So we can think of $\sA^2 \to \sA$ as the universal case of a degeneration.

In the following discussion we allow $X\to B$ to be any of the above.

\subsection{Na\"{\i}ve and universal expanded degenerations}
\label{sec:exp-deg}

 Consider the situation in \ref{Sec:general-deg}. We make definitions analogous to Definitions \ref{Def:standard-pair} and \ref{Def:pair-naive}.

\begin{definition}\label{Def:standard-deg} Let $b$ be a geometric point of $B_0\subset B$ with degenerate fiber $X_b = Y_1\sqcup_DY_2$.
The \emph{standard expansion} of length $\ell\geq 0$ of $X_b$ is a morphism obtained by gluing $\ell$ copies of $P$ to $Y$ along the distinguished divisors:
\begin{equation*}
  X_v[\ell] := Y_1 \mathop{\sqcup}\limits_{D=D_-}  P \mathop{\sqcup}\limits_{D_+=D_-}  \cdots   \mathop{\sqcup}\limits_{D_+=D_-}  P\mathop{\sqcup}\limits_{D_+=D} 
  Y_2
\end{equation*}
where $P = \PP(\cO_{Y_1}(D)|_D \oplus \cO)$.  There is a natural morphism $X_v[\ell]\to X_v$ contracting all the copies of $p$ to $D$. We include the case where  $X_b$ is smooth, declaring the standard expansion of such $X_b$ (necessarily of length 0)  to be the identity $X_b \to X_b$.
\end{definition}

\begin{definition}\label{Def:deg-naive} \begin{enumerate}[leftmargin=0pt,itemindent=3em]
\item A na\"\i ve expansion of $X \to B$ over a base scheme $S$ is a commutative diagram 
$$\xymatrix{
X' \ar[r]\ar[d] & X \ar[d] \\
S \ar[r] & B
}$$
where $X' \to S$ is flat and locally of finite presentation and such that each geometric fiber is  isomorphic over $X$ to a standard expansion.
\item 
We denote by $\fTnaive(X/B,B_0)$ the category of na\"\i ve expansions of $X \to B$, with arrows given by pullback diagrams, and by $X^{\exp}_\naive \to \fTnaive(X/B,B_0)$ the universal na\"\i ve expansion.
\end{enumerate}
\end{definition}

As with Definition \ref{Def:pair-naive} this can be viewed as a category even when $X$ or $B$ is a stack. 
As with Lemma \ref{Lem:exp-basic-comp}, one can pull back naive expanded degenerations along \'etale morphisms $Z \to X$.  As with Lemma \ref{Lem:T-stack}, the definition of arrows in terms of fibered diagrams immediately gives:

\begin{lemma}
The category $\fTnaive(X/B,B_0)$ is a stack.
\end{lemma}

But as with expanded pairs, this is a problematic category in general: see Section~\ref{Sec:trouble}. We instead rely on the universal situation. The following lemma, which  follows directly from the definition, allows us to work with an arbitrary base.

\begin{lemma}\label{Lem:basechangeT} Given a morphism $\phi:\tilde B \to B$, denote $\tilde B_0 =\phi^{-1} B_0$ and  $\tilde X = X \times_B \tilde B$. Then 
$$\fTnaive(\tilde X/\tilde B,\tilde B_0) \simeq \fTnaive(X/B,B_0)\times_B \tilde B.$$
\end{lemma}

Our definition is:
\begin{definition} \label{Def:deg-exp} Denote $\fT := \fTnaive(\sA^2/\sA,\sD)$, with universal family $$(\sA^2)^{\exp} \to \fT.$$ 
\end{definition}

\begin{definition} \label{Def:deg-exp-X}
For a general degeneration $X \to B$ as in \ref{Sec:general-deg}, let $B \to \sA$ be the morphism associated to $(\sL, s)$ and $X \to \sA^2$ associated to $(\sL_1,s_1)$ and $(\sL_2,s_2)$. Given a $B$-scheme $S$, an expansion of $X/B$ over $S$ is a commutative diagram with cartesian square
$$\xymatrix{
X' \ar[r]\ar[d] & (\sA^2)'\ar[d]\\
X\ar[r]\ar[rd] & \sA^2\times_\sA S\ar[d]\\
 & S
}
$$
in which $(\sA^2)'$ is a na\"ive expansion of $\sA^2$ over $S$.
\end{definition}
\begin{remark}
Again we could define a stack  of expansions $\fT(X/B,B_0)$, but by the universal property of fibered products we have a canonical isomorphism
 $\fT(X/B,B_0):= \fT\times_\sA B$,  with universal expansion $$\sX^{\exp} := \sX \times_{\sA^2}  (\sA^2)^{\exp} \to \fT(X/B,B_0).$$ This makes $\fT\times_\sA B$ into the stack of expansions of any degeneration over $B$.
 \end{remark}

\begin{remark}
There is a similarity between expanded pairs and expanded degenerations which can be misleading. Consider a smooth pair $(Y,D)$ over a field $k$. We can define a morphism $Y\to \sA^2$ whose first factor is defined by $(\cO_Y(D), {\mathbf 1}_D)$ and second factor defined by $(\cO_Y(-D), 0)$. This gives a commutative diagram
  \begin{equation*}
    \xymatrix{
      Y \ar[r] \ar[d] & \sA^2 \ar[d]^m \\
      \Spec k \ar[r] & \sA .
    }
  \end{equation*}
  It is tempting therefore to try to view an expanded pair as a special kind of expanded degeneration over $\fT_0 = \Spec k\times_\sA \fT$.  Indeed, the geometric fibers of $Y' = Y\times_\sA (\sA^2)^{\exp}$ are expansions of $(Y, D)$.  However, since $Y \to \Spec k\times_\sA \sA^2$ is not flat, the resulting map $Y'\to \fT_0$ is not flat, so this is not an expanded pair.
\end{remark}

\subsection{Twisted expansions of pairs and degenerations}
In the papers  \cite{AF} one uses auxiliary stack structures on expansions in order to circumvent difficult deformation aspects of the degeneration formula. This formalism is used in \cite{ACW} to compare orbifold Gromov--Witten invariants to relative Gromov--Witten invariants. Here we lay the foundations, which become straightforward using our universal approach. We build on the notation $P$ and $D_\pm$ of \ref{Def:standard-pair}.

\begin{definition}\label{Def:twisted-standard}
\begin{enumerate}[leftmargin=0pt,itemindent=3em]
\item Let $(Y,D)$ be a pair. Fix a non-negative integer $\ell$ and a tuple of  positive integers $\br:= (r_0,\ldots,r_\ell)$. Consider the root stacks $Y_{r_0} = Y(\sqrt[r_0]{D})$ and for $i=1,\ldots \ell$ write $P_i = P(\sqrt[r_{i-1}]{D_-}, \sqrt[r_{i}]{D_+})$. We use the notation $\cD_{i\ \pm}$ for the resulting root divisors;  for $i=1,\ldots \ell-1$ consider the unique isomorphism $\phi_i:\cD_{i\ +} \simeq \cD_{(i+1)\ -}$ which is band-reversing, in the sense that it identifies the normal bundle of  $\cD_{i\ +}$ with the dual of the normal bundle of $ \cD_{(i+1)\ -}$. 

The standard $\br$-twisted expansion of length $\ell$ of $(Y,D)$ is the stack
\begin{equation*}
  Y[\ell,\br]  :=  Y_{r_0}\mathop{\sqcup}\limits_{\cD\mathop{=}\limits_{\phi_0}\cD_-}  P_1 \mathop{\sqcup}\limits_{\cD_+\mathop{=}\limits_{\phi_1}\cD_-}  \cdots   \mathop{\sqcup}\limits_{\cD_+\mathop{=}\limits_{\phi_\ell}\cD_-}  P_\ell.
\end{equation*}

The notation means to convey that $Y_{r_0}$ is joined to $P_1$, and $P_i$ to $P_{i+1}$, by identifying $\cD_{i\ +}$ with $\cD_{(i+1)\ -}$ through $\phi_i$. 
\item 
Similarly, given a degeneration and $b\in B_0$, the standard $\br$-twisted expansion of length $\ell$ of $X_b$ is the stack

\begin{equation*}
  X_b[\ell,\br]  :=  (Y_1)_{r_0}\mathop{\sqcup}\limits_{\cD\mathop{=}\limits_{\phi_0}\cD_-}  P_1 \mathop{\sqcup}\limits_{\cD_+\mathop{=}\limits_{\phi_1}\cD_-}  \cdots   \mathop{\sqcup}\limits_{\cD_+\mathop{=}\limits_{\phi_\ell}\cD_-}  (Y_2)_{r_\ell} 
\end{equation*}
where  $(Y_2)_{r_\ell} = Y_2(\sqrt[r_\ell]{D})$ and the rest of the notation as above.
\end{enumerate}
\end{definition}

We use the notation $\fTnaivetw(X/B,B_0)$ and $\cTnaivetw(Y,D)$ for the stacks of flat families whose fibers are isomorphic to twisted standard expansions. We will use it only briefly in the next section.

We can now define stacks of expansions:

\begin{definition}\label{Def:twisted-exp}
\begin{enumerate}[leftmargin=0pt,itemindent=3em]
\item Denote by $\cTtw$ the category whose objects over a scheme $S$ are flat families of pairs $(\sA',\sD') \to S$ with a morphism $(\sA',\sD')\to (\sA,\sD)$, such that the geometric fibers over $S$ are isomorphic over $(\sA,\sD)$ to standard $\br$-twisted expansions of $(\sA,\sD)$. We call this the stack of twisted expanded pairs, the universal expansion denoted by $(\sA^{\exp\,\tw}, \sD^{\exp\,\tw})$.

For an arbitrary pair, the stack of twisted expansions  of $(Y,D)$ is again $\cTtw(Y,D):= \cTtw$, with the universal expansion $(Y\times _\sA \sA^{\exp\,\tw}, D\times_\cD \cD^{\exp\,\tw})$. 
\item Denote by $\fTtw$ the category whose objects over a scheme $S$ are flat families of pairs $(\sA^2)' \to S$ with a morphism $(\sA^2)'\to \sA^2$, such that the geometric fibers over $S$ are isomorphic over $\sA^2$ to standard $\br$-twisted expansions of the  degeneration $\sA^2\to \sA$ of \ref{Sec:universal-deg}. We call this the stack of twisted expanded degenerations, with universal expansion $(\sA^2)^{\exp\,\tw}$.\Dan{can we find a way around double scripts?}

For an arbitrary degeneration as in \ref{Sec:general-deg} the stack of twisted expansions  of $X$ is again $\fTtw\times_\sA B$, with the universal expansion $X\times _{\sA^2} (\sA^2)^{\exp\,\tw}$. 
\end{enumerate}
\end{definition}

\begin{remark} 
When $(Y,D) = (\sA,\sD)$ it is not hard to see that $P_i$ are actually isomorphic to $P$, and the stack underlying the standard twisted expansion is
$\sA':=\sA \sqcup P\sqcup\cdots \sqcup P$, not depending on the twisting. What varies with the twisting is the morphism  $\sA'\to \sA$.
\end{remark}

\section{The stacks of expansions are algebraic}

\subsection{Semistable maps and the stack of Graber--Vakil}

Recall that a pre-stable marked curve $C$ is called {\em semistable} if any unstable component is rational and has precisely two special geometric points, counting nodes and marked points of $C$.

In \cite{GV}, Graber and Vakil introduce a stack of expanded pairs as follows:
\begin{definition}\label{Def:TGV}
Consider the stack $\fM_{0,3}$ of 3-pointed genus $0$ pre-stable curves and its substack $\fM^\sst_{0,3}$ of semistable curves. We denote the markings by $0,1,\infty$.  Denote by $\TGV$ the substack where the points marked $1,\infty$ lie on the same irreducible component of the curve. 
\end{definition}

The stack $\fM_{0,3}$ is well known to be algebraic and locally of finite type over $\ZZ$; the substacks  $\fM^\sst_{0,3}$ and  $\TGV$  are open, so they are also  algebraic and locally of finite type over $\ZZ$. Graber and Vakil noted that Li's construction of expanded pairs should be independent of the choice of $Y$. Moreover the objects of the stack $\TGV$ are (as we see below) expansions of   $(\PP^1,0)$. The idea is that once we know what to do to the pair $(\PP^1,0)$ we should just ``do the same" to an arbitrary pair. Since an arbitrary pair does not map to $(\PP^1,0)$, we preferred to go by way of the pair $(\sA,\sD)$. This will require just one more step to show that the stack is algebraic. 

In order to work in closer parallel with expanded pairs and degenerations, we go by way of a variant of $\TGV$:

\begin{definition}
Consider the stack $\fM^\tw_{0,1}(\PP^1,1)$ of (not necessarily representable) maps  of degree 1 from 1-pointed twisted prestable curves of genus 0 to $\PP^1$.  The marked point is permitted to have a nontrivial structure.  This stack has a substack $\fM^{\tw\, \sst}_{0,1}(\PP^1,1)$ where the maps are  semistable. Consider the evaluation map  $e:\fM^{\tw\, \sst}_{0,1}(\PP^1,1)\to \PP^1$. Define $$\cTmapstw := e^{-1}\{0\}$$ and $\cTmaps \subset \cTmapstw$ the open substack of untwisted curves.\Dan{verify that this has the right scheme structure} 
\end{definition}

The reader interested only in the untwisted case can safely ignore all references to twisted curves.

Note again that  the stack $\fM^\tw_{0,1}(\PP^1,1)$  is well known to be algebraic and  locally of finite type over $\ZZ$ (see \cite{AOV} for the general  case over $\ZZ$). The stack $\fM^{\tw\,\sst}_{0,1}(\PP^1,1)$ is algebraic as it is an open substack of $\fM^\tw _{0,1}(\PP^1,1)$, and 
$\cTmapstw = e^{-1}\{0\}$ is algebraic as it is closed in  $\fM^{\tw\,\sst}_{0,1}(\PP^1,1)$. The substack $\cTmaps$ is open in the latter, so is algebraic as well.

\begin{lemma}\label{Lem:TP1} We have isomorphisms $\cTnaivetw(\PP^1,D)\simeq \cTmapstw$ and $\cTnaive(\PP^1,D)\simeq \cTmaps$.  Hence $\cTnaivetw(\PP^1,D)$ is an algebraic stack, locally of finite type over $\ZZ$ and $\cTnaive(\PP^1,D)$ an open substack.\Dan{should we avoid the naive thing and go directly to the proof?}
\end{lemma}

\begin{proof}
We provide a base-preserving equivalence of categories. Consider an object $(Y',D') \to (\PP^1,0)$ in $ \cTnaivetw(\PP^1,D)(S)$. Then every geometric fiber is by definition a semistable map, since any unstable component is a root stack of $P \simeq \PP^1$ along its two special points, $D_-$ and $D_+$. Since $D'$ maps to $D$ we have that $e$ is identically $0$. This gives an object of $\cTmapstw$. This identification is clearly compatible with pullbacks, giving a functor $\cTnaivetw(\PP^1,D)\to \cTmapstw$, preserving the twisting.

Similarly, an object of  $\cTmapstw$ is a flat family of twisted pointed curves $(Y',D')$ with a map to $(\PP^1,D)$. We need to verify that the fibers are expansions. We apply induction on the number of irreducible components.  There is a unique component of $Y'$ mapping with degree $1$ onto $\PP^1$; if it is the unique component the map is a root stack at $0$ and we are done. If there is another component, consider an end component on the same fiber of $Y'\to \PP^1$. Since the map is semistable, that end component must contain the marked divisor $D'$ and a unique node, which we identify as $\PP^1_D$. Pruning that component off, we get a new semistable map $(Y'',D'') \to (\PP^1,D)$ which is a twisted expansion by induction. Since $Y'$ is obtained by gluing the end component to $Y''$, and since twisted curves are by definition balanced, we conclude that $(Y',D')$ is a twisted expansion.  Again this identification is compatible with pullbacks, giving a functor $\cTmapstw \to\cTnaivetw(\PP^1,D)$.

The compositions of these functors in either way  are the identity, giving the required equivalence.
\end{proof}

We relate these to $\TGV$ using the following well known result. 
\begin{lemma}\label{Lem:GV-maps}
We have an isomorphism $\TGV\simeq \Tmaps$.
\end{lemma}

\begin{proof}
Given an object $f:(Y',D')\to \PP^1$ in $\Tmaps(S)$ we obtain an object $$\big(Y', D', f^{-1}(1), f^{-1}(\infty)\big)$$ of  $\TGV(S)$, which clearly gives a base-preserving functor $\Tmaps \to \TGV$. Given an object $(Y', p,q,r) \in \TGV(S)$ the invertible sheaf $\cO_{Y'}(r)$ is base-point-free relative to $S$. We obtain an associated morphism  $Y' \to \PP^1_S$ which maps $(p,q,r)$ to  $(0,1,\infty)$, which is an object of $\Tmaps$. It is easy to see that these functors are quasi-inverses, giving the required base-preserving equivalence.
\end{proof}

We now pass to degenerations, where the replacement for the pair $(\PP^1,0)$ requires more work.

\begin{definition}\label{Def:basicdeg} Consider $X$, the blowup of $ \AA^1 \times \PP^1$ at the origin, and $B = \AA^1$, with the projection morphism $X \to B$. The unique singular fiber lies over the origin $0\in \AA^1$ and has two components $Y_i$ intersecting at a unique point $D$.

Consider the stack $\fMtw_{0,0}(X/\AA^1,1)$ of maps of  prestable curves of genus 0 to the fibers of $X/\AA^1$, again of degree 1---that is, with image class equal to the class of the fiber.  Let $\fM^{\tw\,\sst}_{0,0}(X/\AA^1,1)$ be the substack of $\fMtw_{0,0}(X/\AA^1,1)$  parameterizing maps that are semistable and are isomorphisms away from $D$.  \Jonathan{this was incomplete; is this the correct way of completing it?}
\end{definition}

As before, $\fM^\tw _{0,0}(X/\AA^1,1)$ is well known to be algebraic and locally of finite type over $k$, and $\fM^{\tw\,\sst}_{0,0}(X/\AA^1,1)$ an open substack.

Repeating the arguments above we have:

\begin{lemma}\label{Lem:basicfTalg} We have an isomorphism $\fTnaivetw(X/\AA^1,1)\simeq \fM^{\tw\,\sst}_{0,0}(X/\AA^1,1)$.
In particular $\fTnaivetw(X/\AA^1,1)$ is an algebraic stack, locally of finite type over $k$.
\end{lemma}



\subsection{Algebraicity of $\cTtw$ and $\fTtw$} The following proposition implies, with Lemma \ref{Lem:TP1}, that $\cTtw$ is algebraic and locally of finite type over $\ZZ$.  First we define a morphism in one direction in complete generality.

\begin{definition} We define a morphism of stacks  $\Phi^{(Y,D)}: \cTtw \to \cTnaivetw(Y,D)$ by sending an expansion $(\sA',\sD')/S$ of $(\sA,\sD)$ to the fiber product  $(Y',D') := (Y\times_\sA \sA', D\times_\sD \sD')$. We similarly define a morphism of stacks $\Phi^{X/B}: \fT\times_\sA B \to \fTnaivetw(X/B,B_0)$ by sending an expansion $(\sA^2)'$ of $\sA^2\times_\sA B$ to the fibered product $X \times_{\sA^2\times_\sA B} (\sA^2)'$.
\end{definition}

\begin{proposition}\label{Prop:basic-pair-isom}
The morphism $\Phi^{(\PP^1,0)}: \cTtw \to \cTnaivetw(\PP^1,0)$ is an isomorphism. In particular $\cTtw$ is algebraic and locally of finite type over $\ZZ$.
\end{proposition}

\begin{proof}
We provide a quasi-inverse $\Psi^{(\PP^1,0)}: \cTnaivetw(\PP^1,0) \to \cTtw$ as follows. Let $f:(Y',D') \to (\PP^1,0)$ be a twisted expansion over a scheme $S$. Marking $\PP^1$ at $(0, \infty)$ and $Y'$ at $(D',f^{-1}(\infty))$ we can view $f$ as a morphism of~$2$-marked rational curves. By Proposition \ref{prop:equivariant} the morphism $f$ is  $\GG_m$-equivariant with respect to the balanced $\GG_m$ action. Write 
$$\sY' := \bigl[(Y\smallsetminus\{\infty\}) / \GG_m\bigr];\quad \sD': = [D'   / \GG_m].$$
Since $[(\PP^1\smallsetminus\{\infty\}) / \GG_m] = \sA$ we obtain an expansion $(\sY',\sD') \to (\sA, \sD)$ over $S$. The procedure is clearly compatible with pullbacks, giving a morphism $\Psi^{(\PP^1,0)}: \cTnaivetw(\PP^1,0) \to \cTtw$. Checking that the morphisms are quasi-inverses is standard. \Dan{add such check? should we use the lemma \ref{Lem:passtoopen} about opens?}
\end{proof}

Here is the analogue of Proposition \ref{Prop:basic-pair-isom} for twisted degenerations:

\begin{proposition}\label{Prop:basic-deg-isom}
 Consider the degeneration $X \to \AA^1$ of Definition \ref{Def:basicdeg}. The morphism $\Phi^{X/\AA^1}: \fTtw\times_\sA \AA^1 \to \fTnaivetw(X/\AA^1,0)$ is an isomorphism.

The stack $\fT$ is algebraic and locally of finite type over $\ZZ$.
\end{proposition}

\begin{proof}

 We again provide a quasi-inverse $\Psi^{X/\AA^1}: \fTnaivetw(X/\AA^1,0) \to \fTtw\times_\sA \AA^1$, as follows. The proper transform $s_0$ of $0\times \AA^1$ and $s_\infty$  of $\infty\times \AA^1$ are disjoint markings of $X \to \AA^1$. Thus $X \to \AA^1$ is a family of 2-pointed twisted semistable curves, and an object of $\fTnaivetw(X/\AA^1,0)$ gives a morphism of families of semistable curves $X' \to X$.  By Proposition~\ref{prop:equivariant}, this map is equivariant with respect to the canonical actions of $\Gm$ on $X'$ and on $X$.  Taking $\sX = \bigl[ ( X' \setminus(s_0\cup s_\infty) ) /\GG_m\bigr]$ gives an expansion of $\bigl[ ( X \setminus\{0,\infty\} ) /\GG_m \bigr] = \sA^2 \times_\sA \AA^1$. \Dan{details?}

By Lemma \ref{Lem:basicfTalg} the stack $\fTnaivetw(X/\AA^1,0)\simeq \fTtw\times_\sA \AA^1$ is algebraic and locally of finite type over $k$.  Since $\AA^1 \to \sA$ is smooth and surjective it follows that $\fTtw$ is algebraic as well \cite[Lemma C.5]{AOV}.

\end{proof}

With Propositions \ref{Prop:basic-pair-isom}  and \ref{Prop:basic-deg-isom} we have completed the proofs of Theorem \ref{Th:twisted-algebraic} and Part~(1) of Theorem \ref{Th:main-T}.

\subsection{Exotic isomorphisms}\Jonathan{this section should at least contain a statement of the isomorphism $\fM' \simeq \cTtw \times \BGm$.  Also there should be an explicit description of the isomorphism $\cT \simeq \fT$ since it is much simpler than the construction here:  just delete the distinguished divisor}

\subsubsection{The isomorphism between expanded pairs and expanded degenerations}

An expanded pair induces an expanded degeneration by deleting the distinguished divisor.  The only information forgotten this way is the order of twisting along the distinguished divisor:  a distinguished divisor with any desired order of twisting can always be glued onto one end of an expanded degeneration to make an expanded pair.  We make these processes precise in the following proposition.

\begin{proposition} \label{prop:TisT}
There is an isomorphism $\cTtw \simeq \fTtw \times \ZZ_{> 0}$.
\end{proposition}

To construct the isomorphism, it is convenient first to study expanded pairs of length $1$.  Let $\cT'$ be the moduli space of such expansions; it is isomorphic to the substack $\TGV' \subset \TGV$ of twisted curves with at most one node.  

We can describe $\cT'$ and its universal object very explicitly.  The stack $\cT'$ itself is isomorphic to $\sA \times \ZZ_{>0}$ with the second factor keeping track of twisting order along the distinguished divisor.  The universal object $\sZ$ may be defined as the open substack of $\sA^3 \times \ZZ_{>0}$, with coordinates $(x,y,z,n)$, where $y$ and $z$ do not vanish simultaneously; the distinguished divisor is the vanishing locus of $z$.  The projection to the base $\cT' \simeq \sA$ sends $(x,y,z)$ to $xy$ and the projection to the universal pair $(\sA, \sD)$ sends $(x,y,z)$ to $xz^n$.  

There are maps
\begin{equation} \label{eqn:24} \xymatrix@C=0pt{
& \fTtw_{\naive}(\sZ/\sA, \sD) \ar[dr] \ar[dl] \\
\fTtw \times \ZZ_{>0} & & \cTtw .
} \end{equation}
To describe the map on the left, note that $\sZ$ can be viewed as a degeneration over $\sA$, the map $\sZ \rightarrow \sA^2$ coming from $(x,y,z) \mapsto (xz^n,y)$.  By deleting the distinguished divisor, an expansion of $\sZ/\sA$ becomes an expansion of $\sA^2 /\sA$; the map to $\ZZ_{>0}$ simply records the order of twisting along the distinguished divisor.  As for the map on the right, note that if $\sX$ is an expansion of the degeneration $\sZ/\sA$ then by composition we obtain a map $\sX \rightarrow \sZ \rightarrow \sA$.  Comparing the fibers, one sees that these are all expansions of the pair $(\sA,\sD)$.

We show that both maps in Diagram~\eqref{eqn:24} are isomorphisms with the following two lemmas.

\begin{lemma}\label{Lem:passtoopen}
\begin{enumerate}[leftmargin=0pt,itemindent=3em]
\item Let $(Y,D)$ be a pair and $U\subset Y$ open, containing $D$.  Then $\cTnaivetw(Y,D) \simeq \cTnaivetw(U,D)$.  
\item Let $X \to V$ be a degeneration and  $U\subset X$ open, containing $D$.  Then $\fTnaivetw(X/B,B_0) \simeq \fT(U/B_0)$.
\end{enumerate}
\end{lemma}
\begin{proof}
Given a na\"ive expansion $f:(Y',D') \to (Y,D)$ over $S$, we canonically obtain a na\"ive expansion $(U',D') \to (U,D)$ by taking $U' = f^{-1} U$. Given a na\"ive expansion   $f_U:(U',D') \to (U,D)$ over $S$ we have a canonical isomorphism $U' \smallsetminus f_U^{-1} D \simeq U\smallsetminus D$. We obtain a na\"ive expansion $(Y',D') \to (Y,D)$ by gluing $$Y' =U' \mathop\sqcup\limits_{U\smallsetminus D} (Y\smallsetminus D).$$ These provide functors inverse to each other. The construction for expanded degenerations is identical.  
\end{proof}

Applying Part~(2) of the lemma to the inclusion $\sA^2 \subset \sZ$ (as the locus where $z \not= 0$) implies that the map $\fTtw_{\naive}(\sZ/\sA,\sD) \rightarrow \fTtw \times \ZZ_{>0}$ is an isomorphism.

\begin{lemma}
The map $\fTtw_{\naive}(\sZ / \sA, \sD) \rightarrow \cTtw$ is an isomorphism.
\end{lemma}
\begin{proof}
To prove the lemma, we must check that any expanded pair admits a map to an expansion of length $2$.  For this, it is easiest to use the Graber--Vakil perspective on expansions.  Viewing an expanded pair over $S$ as a family of twisted curves $C/S$ with untwisted markings $1$ and $\infty$ and a potentially twisted marking $0$, we use the linear system $\cO_C(0 + 1 + \infty)$ to produce a contraction onto a curve $C'$ that has a marking on every component.  Since $1$ and $\infty$ must lie on the same irreducible component in each fiber of $C'$, the fibers of $C'$ can each have at most one node.
\end{proof}

\subsubsection{Moduli of curves and rigidification}

We describe a second relationship between expansions and semistable curves, this time by \emph{rigidification} (see \cite[Proposition IV.2.3.18]{Giraud} and \cite[Appendix A]{AOVtame}).  

Denote by $\fM' \subset \fM_{0,2}^{\tw\,\sst}$ the open and closed substack of twisted semistable marked curves where the second marking is not twisted. Proposition \ref {prop:equivariant} says that the group scheme $\GG_m$ is a normal subgroup of the inertia stack of $\fM_{0,2}^{\tw\,\sst}$. We can consider the rigidification $\fM'/\cB\GG_m$ (sometimes denoted  $\fM'\!\thickslash \GG_m$) of  $\fM'$. \Dan{switch notation??}

\begin{proposition} \label{prop:pairsiso}
There is a canonical isomorphism $\fM' \xrightarrow{\sim} \cTtw \times \BGm$.
\end{proposition}
\begin{proof}
Let $(C, D, E)$ be an $S$-point of $\fM'$.  That is, $C$ is a family of genus zero curves over $S$ with a potentially twisted marking $D$ and an unitwisted marking $E$.  The complement of the zero section in the normal bundle $N_{E/C}$\Jonathan{normal bundle!} is a $\Gm$-torsor on $S$.  This gives a map $\fM' \rightarrow \cTtw$.  Denote this $\Gm$-torsor by $Q$.

Since $C$ has an action of $\Gm$ (Proposition~\ref{prop:Gm-action}), we can twist $C$ by the opposite torsor of $Q$ to produce a family of semistable curves $C' = C \mathop{\times}\limits^{\Gm} Q^\vee$, with sections $D'$ and $E'$, such that $N_{E' / C'}$ is trivial.  But $N_{E'/C'}$ is isomorphic to the complement of the node in the component of $C'$ containing $E'$.  A trivialization of $N_{E'/C'}$ yields a section of $C'$ confined to the component containing $E'$ and not meeting either $E'$ nor a node.  Calling this section $1$ and letting $\infty = E'$ and $0 = D'$, we get an $S$-point of $\cTtw$.  

The pair $(Q, (C', 0, 1, \infty))$ gives the map claimed by the proposition.  The inverse sends $(Q, (C', 0, 1, \infty))$ to $(Q \mathop{\times}\limits^{\Gm} C', 0, \infty)$.
\end{proof}

\begin{corollary}
There is a canonical isomorphism $\cTtw \simeq \fM' / \BGm$.  Restricting to untwisted objects gives $\cT \simeq \fM_{0,2}^{\sst}/\BGm$.
\end{corollary}

We could now repeat the argument for $\fTtw$, but it is faster to make use of the isomorphism proved in Proposition~\ref{prop:TisT}.  Letting $\fM''$ denote the locus in $\fM_{0,2}^{\tw\,\sst}$ where \emph{both} markings are untwisted and restricting to the locus where the distinguished divisor is untwisted in Proposition~\ref{prop:pairsiso}, we obtain the following corollary.

\begin{corollary}
There is a canonical isomorphism $\fM'' \simeq \fTtw \times \BGm$.  Therefore $\fTtw \simeq \fM'' / \BGm$ and $\fT \simeq \fM_{0,2}^{\sst} / \BGm$.
\end{corollary}

\section{Stacks of non-rigid expansions}\Jonathan{We should also mention the closed embedding of $\cTtw_\sim$ in $\cTtw$}
\setcounter{thm}{0}
\numberwithin{thm}{section}

In \cite{GV} Graber and Vakil showed that localization in relative Gromov--Witten theory necessitates one more stack of expansions. This time we fix only the ``divisor" $D$ and a line bundle $N$, standing in for the normal bundle  $N_{D/ Y}$. We allow twisting: 

\begin{definition}
A standard non-rigid twisted expansion of $(D,N)$ is of the form
$$ P_0 \mathop{\sqcup}\limits_{\cD_+\mathop{=}\limits_{\phi_1}\cD_-}  \cdots   \mathop{\sqcup}\limits_{\cD_+\mathop{=}\limits_{\phi_\ell}\cD_-} P_\ell
$$
where $P = \PP_D(N \oplus \cO_D)$ is the projective completion of $N$, and the notation $P_i$, $\cD_\pm$ and $\phi_i$ is as in Definition \ref{Def:twisted-standard}.

We define $\cT^\tw_\sim$ to be the stack whose objects over $S$ are flat families of  twisted non-rigid expansions of $(\BGm,\cL)$, where $\cL$ is the universal line bundle\Dan{or is it the dual??}.

For an arbitrary $(D,N)$, we define an expansion over $S$ to be a commutative diagram with cartesian squares
$$ \xymatrix{ Z \ar[r]\ar[d] & \sP'\ar[d] \\
 D \ar[r]\ar[rd] & \BGm \times S\ar[d] \\
 & S
}$$ where the right column is a non-rigid twisted  expansion of $(\BGm,\cL)$ and in the middle arrow the factor $D \to \BGm$ comes from the line bundle $N$. 
\end{definition}

Again $\cT^\tw_\sim$ is identified as the stack of non-rigid twisted expansions of $(D,N)$, by the universal property of fibered products.

Graber and Vakil identified the stack of flat families of non-rigid expansions  of $(\Spec k,\cO)$ as
$\fM_{0,2}^{\sst}$. As in the proof of Lemma \ref{Lem:TP1} we have 

\begin{proposition}
There is a canonical isomorphism $\cTtw_\sim\simeq \fM_{0,2}^{\tw\, \sst}$. In particular it is an algebraic stack, locally of finite type over $\ZZ$, and $\cT_\sim\simeq \fM_{0,2}^{\sst}$ is an open substack.
\end{proposition}

\section{Diagonals and automorphisms: na\"ive expansions}
\numberwithin{thm}{subsection}
\setcounter{thm}{0}
\label{Sec:diagonals}
In preparation for proving the remainder of Theorem \ref{Th:main-T} and discussing the shortcomings of na\"\i ve expansions, we need some information on the diagonals of stacks of na\"\i ve expansions and automorphisms of na\"\i ve expansions. Since we will not need twisted versions of these we only discuss the untwisted situation.

\subsection{The diagonals of the stacks of na\"ive expansions}
\begin{proposition} \label{prop:diagonal}
  \begin{enumerate}[leftmargin=0pt,itemindent=3em]
  \item Let $(X,D)$ be a smooth pair with $X$ proper.  The diagonal of $\cTnaivetw(X,D)$ is representable by algebraic spaces and is locally of finite presentation.
  \item Let $X \rightarrow B$ be a proper degeneration.  The diagonal of $\fTnaive(X/B,B_0)$ is representable by algebraic spaces and is locally of finite presentation.
  \end{enumerate}
\end{proposition}
\begin{proof}
  Let $T$ denote $\cTnaive(X,D)$ in the first case of the proposition, and $\fTnaive(X/B,B_0)$ in the second case.  In the first case, let $B$ be a point.  Write $E$ for the universal expansion over $T$.  Let $Y = X \fp_B T$, let $T^2 = T \fp_B T$, and let $Z = X \fp_B T^2$.  Note that there is a proper projection $E \rightarrow Y$.  Let $E^{(i)}$ be the pullback of $E$ to $T^2$ via the $i$-th projection $T^2 \rightarrow T$.  

  To see that the diagonal of $T$ is representable, we must see that the functor
  \begin{equation*}
     \uIsom_Z(E^{(1)}, E^{(2)}) : S \mapsto \Isom_{Z_S}(E^{(1)}_S, E^{(2)}_S)
  \end{equation*}
  is representable by an algebraic space.  Each $E^{(i)}$ is proper over $Z$ and $Z$ is proper over $T^2$, so $E^{(i)}$ is proper, and also flat, over $T^2$. Therefore there is a relative algebraic space $\uIsom_{T^2}(E^{(1)}, E^{(2)})$ whose $S$-points are isomorphisms between $E^{(1)}_S$ and $E^{(2)}_S$.  We can identify\Dan{Please check  -  I did not undersand the formula that was here before} 
  \begin{gather*}
  \uIsom_Z(E^{(1)}, E^{(2)}) = \uIsom_{T^2}(E^{(1)}, E^{(2)}) \fp\limits_{\uHom_{T^2}(E^{(1)},Z\times_{T^2} Z)}\uHom_{T^2}(E^{(1)},Z). 
  \end{gather*}
  In the fiber product the map on the right is induced by the diagonal of $Z$ and the map on the left by the projection $E^{(1)} \to Z$ and the composition $E^{(1)} \to E^{(2)}\to Z$.
  Everything appearing above is an algebraic space over $T^2$ because $E^{(i)}$ are proper and flat over $T^2$ (cf.\ \cite[Theorem~6.1]{Artin-F1}).

  It follows immediately from the fact that $E^{(1)}$ and $E^{(2)}$ are locally of finite presentation over $Z$ that $\uIsom_Z(E^{(1)}, E^{(2)})$ is locally of finite presentation over $T$ (see \cite[Proposition~(8.13.1)]{ega-4-3}).
\end{proof}

\subsection{Automorphisms of na\"\i ve expansions}

\begin{lemma} \label{lem:accordion-auts} \Jonathan{isn't there a reference to Li somewhere?  Maybe \cite[Corollary~1.4]{Li1}}
\begin{enumerate}[leftmargin=0pt,itemindent=3em]
 \item  An isomorphism $(Y'_1, D'_1) \rightarrow (Y'_2, D'_2)$ of expansions of $(Y,D)$ over an algebraically closed field  is given by an order preserving bijection between the sets of irreducible components and an isomorphism on each component isomorphic to $\PP(N_{D/Y} \oplus \cO_D)$ preserving the fixed points of the $\GG_m$ action and structure map. The same holds for the case of expanded degenerations.\Dan{rephrase degeneration case?}
\item   If each irreducible component of an expansion of $(Y,D)$, except $Y$, is identified with $\PP(N_{D/Y} \oplus \cO_D)$ then  automorphisms of that component preserving fixed loci and structure map are in bijection with $\Gm(D)$. The same holds for the case of expanded degenerations.
\item  The automorphism group of an expansion of length~$\ell$ is $\Gm(D)^\ell$.
\end{enumerate}
\end{lemma}

\begin{proof}
Any isomorphism of pairs induces a bijection between the sets of irreducible components.  In this case, the order must be preserved because the divisors $D_1$ and $D_2$ must be respected, as must the components isomorphic to $Y$. 
The automorphism group of $\PP(N_{D/Y} \oplus \cO_D)$ over $D$ is the same as the automorphism group of the underlying $\Gm$-torsor $N_{D/Y} \smallsetminus 0$, which is just $\Gm(D)$.  
\end{proof}

\subsection{Comparison of geometric objects}

\begin{definition}\label{Def:Phi}
\begin{enumerate}[leftmargin=0pt,itemindent=3em] \item
Let $(X,D)$ be a smooth pair.  Define  $\Phi = \Phi^{(X,D)} : \cT \rightarrow \cTnaive(X,D)$ by sending an expansion $(\sA',\sD')/S$ of $(\sA,\sD)$ to $(X',D') := (X\times_\sA \sA', D\times_\sD \sD')$. 
\item
Let $(X/B,B_0)$ be a degeneration.  Define  $\Phi = \Phi^{X/B} : \fT\times_\sA B \rightarrow \cTnaive(X/B,B_0)$ by sending an expansion $(\sA^2)'/S$ of 
$(\sA^2 \fp_{\sA} B) / B$ to $X' := X\times_{\sA^2\times_\sA B} (\sA^2)'$. 
\end{enumerate}
\end{definition}

\begin{proposition}\label{Prop:geometric-points} Assume $D$ is smooth, proper and connected. Then the functors $\Phi^{(Y,D)}$ and $\Phi^{X/B}$ induce equivalences on geometric points.
\end{proposition}

\begin{proof}
  Indeed, on both sides the geometric objects are expansions, which are indexed up to isomorphism by non-negative integers in either case.  Furthermore, the automorphism group of an expansion of length $n$ over an algebraically closed field $k$ is $\Gm(k)^n$ (as we saw in Lemma~\ref{lem:accordion-auts}) both in the stack of expansions and the stack of na\"\i ve expansions.  The map $\Gm(k)^n \rightarrow \Gm(k)^n$ is an isomorphism, so the functor is also fully faithful as required.
  \end{proof}

\subsection{The trouble with na\"\i ve expansions}\label{Sec:trouble}

Throughout the discussion so far, we have used na\"\i ve expansions as a building block of our stacks of expansions. The stacks $\cT$, $\fT$ and $\cT_\sim$ themselves are defined as stacks of na\"\i ve expansions of a universal object. Li's stacks will be studied as substacks of the corresponding stacks of na\"\i ve expansions. And the proofs of the main results use them in several steps.

So why not define the stacks of expansions using na\"\i ve expansions and be done with it? There are several issues that arise with $\cTnaive(Y,D)$ and $\fTnaive(X/B,B_0)$. We describe in the case of the stack of pairs, the case of degenerations being similar.

\subsubsection{Deforming the normal bundle.} First, there are problems with deformation spaces which prevent these stack from being algebraic in general. Assume $Y$ is smooth and projective, $D$ is smooth an irreducible, but assume $H^1(D, \cO_D) \neq 0$. Then the normal bundle $N_{D/ Y}$ has a non-constant deformation to another line bundle $L$. It follows that the na\"\i ve expansion $Y' = Y \sqcup_D P$ admits a non-constant deformation to $Y'' = Y \sqcup_D\ \PP_D(\cO\oplus L)$, and the latter is not an expansion.

Assuming $H^1(D, \cO_D) = 0$ would clearly be too restrictive. Jun Li overcame this issue by defining a stack where local models are fixed, excluding such deformation by definition.  
We would be content if this covered all the situations we care about, but this is not the case.

\subsubsection{Disconnected divisor} Consider now the case where $Y$ is smooth and projective, $D$ is smooth, $H^1(D, \cO_D) = 0$ but $D$ decomposes as $D_1\sqcup D_2$.  Then the expansion $Y' = Y \sqcup_D P_D$ admits a non-constant deformation to $Y'' = Y \sqcup_{D_1} P_{D_1}$, so again the stack of na\"\i ve expansions is not algebraic. Li's approach still gives us a good stack of expansions; the logarithmic approach however diverges - it gives us the stack $\cTnaive(Y,D_1) \times \cTnaive(Y,D_2).$ Of course, one can work Li's 
solution, but things can still get a bit more complicated.
\Dan{Find a place where this situation does arise in the literature} 

\subsubsection{Non-proper divisor} Now drop the assumption that $Y$ and $D$ are proper, so $D$ could be affine of positive dimension. Then the automorphism group of a na\"\i ve expansion is no longer finite dimensional, and the diagonal of the stack of expansions is not locally of finite presentation. This case is not resolved by either Li's or Kim's approach. The most common situation where this occurs in the literature is when $Y$ is the total space of a vector bundle over a projective variety $Y_0$,and $D$ is the pullback of a divisor $D_0$ on $Y_0$. 

Of course, such cases can be treated directly, for instance by pulling back expansions of $(Y_0, D_0)$, which amounts to an ad-hoc application of the principle used in this paper. Our approach using the universal pair $(\sA, \sD)$ allows one to expand any pair without restriction and without need for ad-hoc constructions.

\section{Comparison with Li's approach} \Jonathan{section numbering was messed up in this section; is the change okay?}
\label{sec:TLi}

\subsection{Expanded degenerations}
\label{sec:exp-deg-li}

Throughout this section we assume that $B$ is a smooth complex curve, $B_0 = \{b_0\}$ a point, $X$ is smooth,  $X/B$ is projective and $D$ is  smooth and irreducible.\Dan{fix notation of degenerations vs pairs!!}

\begin{corollary}[{of Lemma~\ref{lem:accordion-auts}; cf.\ \cite[Corollary~1.4]{Li1}}] \label{cor:accordion-auts}
  If $X$ is proper and $D$ is connected then the automorphism group of an expansion of $(X,D)$ of length~$\ell$ is $(\CC^\ast)^\ell$.
\end{corollary}

A family of expanded degenerations over a base scheme $S$ is a family of schemes over $S$ that is isomorphic, \emph{\'etale}-locally in $S$, to one of a collection of standard models.  We recall the definition of the standard models below.

Suppose that $\pi : X \rightarrow B$ is a flat family whose general fiber is smooth and whose fiber over $0$ is the union of two smooth schemes along a smooth divisor.  J. Li constructs a collection of standard models 
\begin{equation} \label{eqn:18}
  \xymatrix{
    X[\ell] \ar[r] \ar[d] & X \ar[d] \\
    B[\ell] \ar[r] & B
  }
\end{equation}
for expansions of $X_0$.  
\begin{definition}[{\cite[Definition~1.9]{Li1}}] \label{def:Li-deg}
  Let $\fTLi(X/B,b_0)$ be the fibered category whose $S$-points are families $Z \rightarrow S$ such that, \'etale-locally in $S$, there exists a map $S \rightarrow \AA^{\ell + 1}$ and an isomorphism $Z \cong S \fp_{\AA^{\ell + 1}} X[\ell]$.  Morphisms between such families are morphisms of schemes that commute with the projections to $S \times X$.
\end{definition}

It will follow from the equivalence proved in Theorem \ref{Th:main-T} that up to canonical equivalence, the above definition of $\fTLi(X/B,b_0)$ does not depend on $X$. 
We summarize J.\ Li's construction of the standard models below, taking a slightly different perspective; see \cite[Section~1.1]{Li1} for more details.

For Li's construction, it is enough to work in a small neighborhood of $b_0 \in B$.  We can therefore assume that there is an \'etale map $B \rightarrow \AA^1$ whose fiber over $0 \in \AA^1$ is $b_0 \in B$.  We will construct the standard models $X[\ell] \rightarrow \AA^1[\ell]$ of $X \rightarrow \AA^1$, and then we can define $B[\ell] = \AA^1[\ell] \fp_{\AA^1} B$, through which the map $X[\ell] \rightarrow \AA^1[\ell]$ must factor.    \Jonathan{there's a notation conflict here:  the $\ell$-th expansion of $X$ over $\sA[\ell] = \sA^{\ell}$ is the quotient of the $\ell$-th expansion of $X$ over $\AA[\ell] = \AA^\ell$ by the action of $\Gm^{\ell}$, but the notation $X[\ell]$ is used for both...}

\begin{remark}
  It is possible to avoid having to pass to an \'etale neighborhood of $b_0$ by taking the base of the degeneration to be $\sA = [\AA^1 / \Gm]$ instead of $B$.  The Cartier divisor $b_0$ gives a map $B \rightarrow \sA$, and expansions of the base $B$ are obtained by base change from expansions of $\sA$:  we have $B[\ell] = B \fp_{\sA} \sA[\ell]$ and $\sA[\ell] = \sA^\ell \rightarrow \sA$ is the multiplication map.  

  We leave it to the reader to verify that the discussion in this section carries over word for word replacing the symbol $\AA$ by $\sA$ (and $\AA^\ell$ by $\sA^\ell$).\Dan{move this discussion to the section about limits}
\end{remark}

We will have $\AA^1[\ell] = \AA^{\ell + 1}$ and the projection $\AA^1[\ell] = \AA^{\ell + 1} \rightarrow \AA^1$ sends $(t_0, \ldots, t_{\ell})$ to $\prod t_i$.  The construction of $X[\ell]$ is inductive.  Assume that we have a construction of $\pi_1^X : X[1] \rightarrow \AA^1[1] = \AA^2$ for any family $X \rightarrow \AA^1$ of the type described in Section \ref{Sec:basic-deg}.  We will see that composing $\pi_1$ with the first projection gives a new family $p_1 \circ \pi_1 : X[1] \rightarrow \AA^1$ as in \ref{Sec:basic-deg}.  Expanding this family using the same construction, we obtain a new family $X[1][1] \rightarrow \AA^2$ and a commutative diagram
\begin{equation*}
  \xymatrix{
    X[1][1] \ar[r]^{\pi_1^{X[1]}} \ar[dd] & X[1] \ar[d] \\
    & \AA^1[1] \ar[d]^{p_1} \\
    \AA^1[1] \ar[r]^m & \AA^1
  }
\end{equation*}
where $m : \AA^1[1] = \AA^2 \rightarrow \AA^1$ is the multiplication map and $p_1 : \AA^1[1] = \AA^2 \rightarrow \AA^1$ is the first projection.  We can identify $(\AA^1[1], m) \fp_{\AA^1} (\AA^1[1], p_1)$ with $\AA^3 = \AA^1[2]$ and defining $X[2] = X[1][1]$, we obtain from the diagram above a map $X[2] \rightarrow \AA^1[2]$.  Composing 
\begin{equation*}
  X[2] = X[1][1] \rightarrow \AA^1[2] = \AA^3 \xrightarrow{p_1} \AA^1
\end{equation*}
we obtain a new family as in \ref{Sec:basic-deg}.  Inductively, we obtain $X[\ell+1] = X[\ell][1] \rightarrow \AA^1[\ell] = \AA^{\ell + 1}$, exactly as above.

It remains to explain the construction of $X[1]$ and verify that the composition $X[1] \rightarrow \AA^1[1] \xrightarrow{p_1} \AA^1$ is a family of the type described in Section~\ref{Sec:basic-deg}.  First, let $X'$ be the base change of $X \rightarrow \AA^1$ by the multiplication map $m : \AA^2 \rightarrow \AA^1$ sending $(t_1, t_2) \mapsto t_1 t_2$.  The total space of $X'$ is singular along $D \times (0,0)$.  We construct $X[1]$ as a small resolution of $X'$.

Let $Y_1$ and $Y_2$ be the two irreducible components of $X_0$.  Each is a divisor in $X$.  Let $(\cO_X(Y_i), y_i)$ be a line bundle and section such that the vanishing locus of $y_i$ is $Y_i$.  These are determined up to unique isomoprhism by $Y_i$.

Let $P$ be the $\PP^1$-bundle $\PP(\cO_{X'} \oplus m^\ast \cO_X(Y_2)) \cong \PP(m^\ast \cO_X(Y_1) \oplus \cO_{X'})$ where $m : X' \rightarrow X$ (abusively) denotes the projection.  If $(t_1, t_2)$ denote the coordinates on $\AA^2$, then $(t_1, y_1)$ is a section of $P$ over the complement of $D \times (0,0)$ in $X'$.  Let $X[1]$ be the closure of the image of this section in $P$.

\begin{remark}
  \'Etale locally $X'$ admits a smooth morphism  to the locus $\set{y_1 y_2 = t_1 t_2}$ inside $\AA^4$ and $X[1]$ resolves the indeterminacy of the rational map $(y_1, t_1) : X' \rightarrow \PP^1$.
\end{remark}

Composition of the projections $X[1] \rightarrow X'$ and $X' \rightarrow \AA^2$ gives us a map $\pi_1 : X[1] \rightarrow \AA^2$.   Away from $0 \in \AA^1$, the fiber of $p_1 \pi_1 : X[1] \rightarrow \AA^1$ can be identified with the projection  $X \times (\AA^1 \smallsetminus 0)^2 \rightarrow (\AA^1 \smallsetminus 0)$ on the last component.  It is therefore smooth away from $0$.  Over $0 \in \AA^1$, the fiber is the union of two smooth schemes along a smooth divisor (one is $Y_2$ and the other is the blow-up of $Y_1 \times \AA^1$ along $D \times 0$, where $D$ is the divisor along which $Y_1$ is joined to $Y_2$ in $X_0$).  We may therefore iterate the construction to obtain the standard models.

\begin{remark}\label{Rem:Li-etale}
  In \cite[Definition~1.9]{Li1}, families of expansions are required to be isomorphic to the standard models over an open cover.  This open cover must be taken in the analytic topology since, as the following variant of Hironaka's famous example demonstrates, Zariski open covers do not suffice to obtain a stack in the \'etale topology.  This is why we have insisted on gluing in the \'etale topology above.\Jonathan{I pared this down a little; is it okay?}

  Let $\tS$ be a pair of rational curves joined at two nodes.  Label the nodes $\ts_1$ and $\ts_2$ and the components $\tS_1$ and $\tS_2$.  This carries a free action of $\ZZ / 2 \ZZ$ which exchanges the components and the nodes.  Let $S$ be the quotient, which is a rational curve joined to itself at a single node.

  Over $\tS$ we construct an equivariant family $\tY$ of expansions of $X = \PP^1 \cp_{\rm point} \PP^1$.  On the complement of $\ts_1$, we take the family induced from the standard model $X[1]$ by the inclusion of $\tS - \ts_1$ as the coordinate axes in $\AA^2$.  Over $\tS - \ts_2$, we take the family induced the same inclusion, but with the axes reversed.

  Gluing these together, we get a family over $\tS$ whose fiber over $\ts_1$ has $3$ components.  Passing from $\ts_1$ to $\tS_1$ smoothes the second node in the fiber over $\ts_1$ so that a generic fiber of $\tY$ over $\tS_1$ has two components.  Passing to $\ts_2$ degenerates the first of these components to the union of two components, so that the fiber over $\ts_2$ has $3$ components.  Passing now to $\tS_2$ smoothes the second node, and passing to $\ts_1$ degenerates the first component.

  There is a free action of $\ZZ / 2 \ZZ$ on $\tY$  making the projection to $\tS$ equivariant.  Let $Y \rightarrow S$ be the quotient.  This family is not induced in \emph{any} Zariski neighborhood $S^\circ$ of the node from any $X[n]$ via a map $S^\circ \rightarrow \AA^\ell$.  Indeed, if it were then $Y$ would be quasi-projective.  But if $m_1, m_2, m_3$ are the degrees of the restriction of an ample line bundle to the fiber over the node, we get
  \begin{gather*}
    m_1 + m_2 = m_1 \\
    m_2 + m_3 = m_3 ,
  \end{gather*}
  so any map to a projective space must collapse the middle component.  Thus this family does not appear in the \emph{Zariski} stack.

  This point was brought up and clarified in discussions with Michael Thaddeus.
\end{remark}

\subsection{Expanded pairs}
\label{sec:exp-pairs-li}

J. Li constructs the moduli space of expanded pairs  by reduction to the construction for expanded degenerations:  the first expanded pair is a degeneration of the kind studied in the last section.

If $(X, D)$ is a smooth pair, the first expansion $X[1]$ is obtained by blowing up $X \times \AA^1$ at the locus $(D, 0)$.  This gives a family over $\AA^1$ whose fiber over~$0$ is the union of $X$ and $\PP(N_{D/X} \oplus \cO_D)$ along $D$.  The proper transform of $D \times \AA^1$ is a divisor in $X[1]$.

Expanding $X[1] \rightarrow \AA^1$ as in the last section, we get the standard models $X[\ell] \rightarrow \AA^\ell$.  Note however that there is a shift of index as compared to the standard models for degenerations.  We use the notation $\cTLi(X,D)$ for the stack defined this way. 

\subsection{Algebraicity and comparison}
We note the canonical embeddings which allow us to compare our stacks with   Li's stacks.
\begin{definition} 
Let $\cTLi(X,D) \hookrightarrow \cTnaive(X,D)$ and $\cTLi(X/B,b_0) \hookrightarrow \cTnaive(X/B,b_0)$ be the embeddings given by viewing Li's objects as na\"ive expansions and arrows as morphsims of na\"ive expansions.\end{definition}


\begin{proposition} \label{prop:TLi-algebraic}
  The stacks $\cTLi(X,D)$ and  $\fTLi(X/B,B_0)$ are algebraic.
\end{proposition}
\begin{proof}
To treat degenerations and pairs together, we use a common notation as in the proof of Proposition \ref{prop:diagonal}:  $X$ is either the total space of a degeneration or a pair, $X'$ is an expansion, and $T_{\Li}$ and $T_\naive$ are the relevant stacks of expansions. 

  The diagonal of $T_{\Li}$  is pulled back from the diagonal of $T_\naive$ via the embedding $T_\Li\hookrightarrow T_\naive$. The diagonal of $T_\naive$ is representable by Proposition~\ref{prop:diagonal}.
  
  On the other hand, the map $\AA^n \rightarrow T_\Li$ induced from the standard models is formally smooth by definition \cite[Definition~1.9]{Li1}:  any object of $T_\Li(S)$ is required to lift, \'etale locally in $S$, to an object of $\AA^n$.  Applying this to an infinitesimal extension of $S$  implies the formal criterion for smoothness.  Furthermore, $\AA^n \rightarrow T_\Li$ is locally of finite presentation, because $\AA^n$ is locally of finite presentation and the diagonal of $T_\Li$ is locally of finite presentation.  Since $\AA^n \rightarrow T_\Li$ is also representable, this implies that it is smooth and since the maps $\AA^n \rightarrow T_\Li$ cover $T_\Li$, this implies that $T_\Li$ is an algebraic stack.
\end{proof}

The following proposition completes the proof of Theorem \ref{Th:main-T}.

\begin{proposition}
 Assume $D$ is smooth, projective, and connected. Then the map $\Phi^{(Y,D)}$ of Definition \ref{Def:Phi} factors through an equivalence $\cT \rightarrow \cTLi(X,D)$ (for a pair) and the map $\Phi^{X/B}$ factors through  $\fT\times_\sA B \rightarrow \fTLi(X/B,b_0)$ (for a degeneration).
\end{proposition}
\begin{proof} 

First consider the case of pairs.

To see that the factorization exists, note that we have commutative diagrams
\begin{equation*} \xymatrix{
\AA^n \ar[dr] \ar[d] \\
\cT \ar[r]_<>(0.5){\Phi} & \cTnaive(X,D) 
} \end{equation*}
coming from Li's standard models so it will be sufficient to show that the maps $\AA^n \rightarrow \cT$ form a smooth cover.  Since $\cT \simeq \TGV$, it is enough to show that $\AA^n \rightarrow \TGV$ is a versal deformation space for the object $Q$ over $0 \in \AA^n$.  Recall that $Q$ is a chain of $n+1$ rational curves with two marked points on the first component and one marked point on the last component.  The $n$ directions in $\AA^n$ correspond to deforming the $n$ nodes of the chain, and this is well known to be a versal deformation space.

  Note that the automorphism group of $Q$ is $\Gm^n$, acting on $\AA^n$ by scaling the coordinates.  Therefore we obtain an \'etale morphism $\sA^n = [ \AA^n / \Gm^n ] \rightarrow \cT$.

  The composition $\AA^n \rightarrow \cT \rightarrow \cTLi(X,D)$ is also smooth, as we saw in the proof of Proposition~\ref{prop:TLi-algebraic}.  By Lemma~\ref{lem:accordion-auts}, the automorphism group of the image of $0 \in \AA^n$ in $\cTLi(X,D)$ is $\Gm^n$, so there is an induced \'etale map $\sA^n = [ \AA^n / \Gm^n ] \rightarrow \cTLi(X,D)$.  Therefore $\sA^n$ is \'etale over both $\cT$ and $\cTLi(X,D)$.  This implies that the map $\cT \rightarrow \cTLi(X,D)$ is \'etale.\Dan{Streamline!!!}

  Finally, we note that $\cT \rightarrow \cTLi(X,D)$ induces an equivalence on geometric points by Proposition \ref{Prop:geometric-points}.  
  
Next we consider degenerations.  It is possible to deduce that $\Phi^{X/B}$ an isomorphism formally from the isomorphism between the moduli spaces of expanded pairs and degenerations (Proposition~\ref{prop:TisT} and Section~\ref{sec:exp-pairs-li}), so we will only sketch the geometric argument parallel to the one given above for pairs.  

We show that $\Phi^{X/B}$ induces an equivalence $\fT \fp_{\sA} B \rightarrow \fTLi(X/B,b_0)$ is an equivalence in the case $B = \AA^1$, with the general case following by gluing and base change.  To show that $\Phi^{X/B}$ factors through $\fTLi(X/B,b_0)$, we use the isomorphism $\fT\times_\sA \AA^1 \simeq  \fM^{\sst}_{0,0}(X/\AA^1,1)$.  We obtain this time that $\AA^{n+1} \to \fT\times_\sA \AA^1$ is versal as before. The automorphism group  is  $\Gm^{n}$, giving an  \'etale map $ [ \AA^{n+1} / \Gm^{n} ] \rightarrow \fTLi(X/B,b_0)$, and  $\fT\times_\sA B \rightarrow \TLi(X/B,b_0)$ is therefore \'etale. Again the map is an equivalence on geometric objects, completing the argument in this case as well.\Dan{too sketchy}\Jonathan{still needs work; could we use the fact that $\fT \simeq \cT$ somehow to omit the argument?}
\end{proof}

This completes the proof of our Comparison Theorem \ref{Th:main-T}.

\section{Labels and twists}
\label{sec:variants}

\subsection{Labelling by a set}
\label{sec:label-by-set}

It is sometimes desirable---in~\cite{AF}, for example---to work with variants of the stack $\cT$ or $\fT$ which include labels for  the nodes of the fibers.  

Let $F$ be the non-smooth locus of the projection from the universal expanded degeneration to $\fT$.  Let $\sS$ be a set, which we will call the set of labels.  Let $\fT^{\sS}$ be the sheaf on $\fT$ whose sections over $S$ are the locally constant functions from $F_S$ to $\sS$.  This is an \'etale sheaf on $\fT$, so its total space, denoted $\fT^{\sS}$, is a stack equipped with a formally \'etale morphism to $\fT$.

By definition, objects of the stack $\fT^{\sS}$ are expansions of the degeneration $\sA^2 \rightarrow \sA$ together with a continuous labelling of the non-smooth locus in the set $\sS$, and arrows are arrows of expansions preserving the labelling.  The universal expansion $\sX^{\sS}$ is the pullback, via the projection $\fT^{\sS} \rightarrow \fT$ of the universal expansion $\sX$ of $\fT$.

We will also make use of $\cT^{\sS}$, the stack of expanded pairs with a labelling in $\sS$.  In this case, let $F$ be the union of the non-smooth locus of the universal expansion of $(\sA,\sD)$ {\em and its distinguished divisor}.  Let $\cT^{\sS}$ be the \'etale sheaf on $\cT$ whose sections over $S$ are the locally constant maps from $F_S$ to $\sS$.  Since $F \cong F_0 \sqcup \sD$, where $F_0$ is the non-smooth locus of the universal expansion of $(\sA,\sD)$, and since  $\cT \cong \fT$ by Proposition \ref{prop:TisT}, we can identify
\begin{gather*}
  \cT^{\sS} \cong \fT^{\sS} \times \sS .
\end{gather*}
The second factor records the labelling of the distinguished divisor.  We denote the universal family over $\cT^\sS$ by $\cX^\sS$.

\begin{lemma}
  Let $T$ be either $\fT$ or $\cT$ and let $T^{\sS}$ be $\fT^{\sS}$ or $\cT^{\sS}$, respectively.  The projection $T^{\sS} \rightarrow T$ is representable by algebraic spaces. 
\end{lemma}

\begin{proof} The morphism is formally \'etale by the discussion above; it is locally of finite presentation over $T$ because both $F$ and $\sS \times T$ are \cite[Proposition~8.13.1]{ega-4-3}. 

We have to check that formal sections algebraize uniquely. \Dan{why do we need this??} \Jonathan{a functor that is \'etale and locally of finite presentation need not be algebraic if it doesn't satisfy algebraization}
Suppose that $S$ is the spectrum of a complete noetherian local ring and $S_0$ is its closed point.  Then the connected components of $F_{S_0}$ are the same as the connected components of $F_S$, so locally constant functions on $F_{S_0}$ algebraize uniquely to locally constant functions on $F_S$.  \end{proof}

Therefore $\cT^\sS$ and $\fT^\sS$ are algebraic:

\begin{proposition}\label{Prop:labels-algebraic}
  There are algebraic stacks $\fT^{\sS}$ and $\cT^{\sS}$, which respectively parameterize expanded degenerations with nodes labelled by $\sS$ and expanded pairs with nodes and distinguished divisor labelled by $\sS$.
\end{proposition}

\subsection{Twisting choice} 
In this section we work over a field of characteristic 0. We will now relate the stack $\cT^\tw$ to Olsson's log-twisting construction.  
\label{sec:twisting-choice} 
It is convenient to use the formalism of Borne and Vistoli \cite{BV} for logarithmic structures. We use the notation $\bN = \ZZ_{>0}$. 

The stack $\cT$  admits a canonical locally free logarithmic structure $(\cT,M_\cT)$ coming from the normal crossings divisor over which the universal expansion fails to be smooth, smilarly for $(\fT,M_\fT)$. Pulling back, this gives a locally free logarithmic structure on $\cT^{\sS}$ and $\fT^\sS$ for any set~$\sS$.  

Following~\cite{BV}, we think about a logarithmic structure as a homomorphism from an \'etale sheaf of monoids $\oM$ into $\sA$.  Put $T = \cT$ or $T = \fT$ as the case warrants and write $T^{\sS}$ for its labelled analogue.  Define $F$ as in Section~\ref{sec:label-by-set}.  We note that, \'etale locally in a $T$-scheme $S$, the connected components of $F$ are in bijection with the generators of $\oM_T$.  Therefore we can use a map $r : F \rightarrow \bN$ to construct a new sheaf of monoids $\oM(r)$ that adds an $r(z)$-th root to the generator corresponding to a connected component $z$ of $F$ whenever $z$ corresponds to a singularity. 

Now, consider a map $\fr : \sS \rightarrow \bN$; we will call it a \emph{twisting choice} because of the way we will use it below.  Composing with the tautological map $F\to \sS$ on $T^{\sS}$, we obtain a map $F \rightarrow \bN$.  We may then construct a sheaf of monoids $\oM(\fr)$ on~$T^{\sS}$, as above.
\begin{definition}\label{Def:r-twisted-T} Fix a twisting choice $\sS \to \bN$.
Denote by $\cT^{\fr}$ the stack of  objects of $\cT^\sS$ together with  extensions of the map $\oM_\cT \rightarrow \sA$ to a map $\oM_\cT(\fr) \rightarrow \sA$. Similarly, denote by $\fT^{\fr}$ the stack of objects of $\fT^\sS$ together with extensions of the morphism $\oM_\fT \rightarrow \sA$ to a map $\oM_\fT(\fr) \rightarrow \sA$.
\end{definition}

\begin{proposition}
  The maps $\cT^{\fr} \rightarrow \cT^{\sS}$ and $\fT^{\fr} \rightarrow \fT^{\sS}$ are of Deligne--Mumford type and are isomorphisms on dense open substacks of source and target.
\end{proposition}
\begin{proof}
  Since $\oM_{\cT}$ is trivial on the dense open point of $\cT$ the map $\cT^{\fr} \rightarrow \cT^{\sS}$ is an isomorphism there.  To see that the map is of Deligne--Mumford type, we can work \'etale locally on an $S$-point of $\cT^{\sS}$ and assume that the logarithmic  structure is free, at which point $\cT^{\fr} \fp_{\cT^{\sS}} S$ becomes a root stack, which is Deligne--Mumford because we work in characteristic 0. The proof for $\fT^\fr$ is identical.
\end{proof}

The pre-image of the divisor that was used to define the logarithmic structure on $\fT^\sS$ is a normal crossings divisor in the universal expansion $\sX^\sS$.  It therefore induces a locally free logarithmic structure on $\sX^\sS$.  \'Etale locally in $\fT$ (i.e., \'etale locally in $S$ for each $S$-point of $\fT$) each generator of $\oM_{\cT}$ maps, away from $F$, to a generator of $\oM_{\sX}$, and to a sum of two generators near $F$. Define $\oM_{\sX}(\fr)$ to be the sheaf of monoids obtained by adjoining an $\fr(z)$-th root to each of these generators. 

For the universal family $\cX^\sS$ over $\cT^\sS$ we apply the same procedure with the logarithmic structure associated to the normal crossings divisor associated to both singular locus and marking.


\begin{definition}\label{Def:r-twisted-X}  Denote by $\sX^{\fr}$ over $\sX$ the moduli stack of extensions of the map $\oM_{\sX} \rightarrow \sA$ to $\oM_{\sX}(\fr) \rightarrow \sA$; and similarly for $\cX^{\fr}$ over $\cX$. We call these the stacks of $\fr$-twisted  expansions. 
\end{definition}

\begin{lemma} \label{lem:tw-fibers}
\begin{enumerate}[leftmargin=0pt,itemindent=3em]
 \item The fibers of $\sX^{\fr}$ over $\fT^{\fr}$ are twisted expansions of $\sA^2$ over $\sA$. The  fibers of $\cX^{\fr}$ over $\cT^{\fr}$ are twisted expansions of $\sA$.
 \item The resulting morphism $\fT^{\fr} \to \fT^\tw\times_\fT \fT^\sS$ or  $\cT^{\fr} \to \cT^\tw\times_\cT \cT^\sS$ is an isomorphism onto the open substack consisting of pairs $({X^\tw}', \phi:F\to \sS)$, with ${X^\tw}'$ a twisted expansion and $ \phi:F\to \sS$ a labelling as before, such that the order of the stabilizer of a point $z$ lying over $\bar z \in F$ is $\fr(\phi(\bar z))$.
 \item When $\sS = \bN$ and $\fr: \bN\to \bN$ is the identity, we have $\fT^{\fr} = \fT^\tw$ and $\cT^{\fr} = \cT^\tw$.
 \end{enumerate}
\end{lemma}
\begin{proof}
We can check the statements after pulling back $\cX^{\fr} \to \sA$ via the map $\AA^1 \to \sA$,  and similarly pulling back $\sX^{\fr} \to \sA^2$ via the map $X \to \sA^2$. Also the third statement follows from the second.

According to \cite{Olsson}, \cite{AOVmaps} there is an equivalence of categories between log twisted curves and twisted curves. Thus the stack obtained from a nodal curve $C/S$, with canonical log structure $\oM_C\to \sA$ over $\oM_S\to \sA$ by imposing a simple extension of the log structure $\oM_C \subset \oM_C' \to \sA$ compatible with $\oM_S\subset \oM_S' \to \sA$,  is a twisted curve, and every twisted curve is canonically obtained in this manner. The order of the stabilizer of a point $z$ corresponding to the $i$-th component of the locally free log structure is precisely the orther of the $i$-th component of the quotient $\oM_C' / \oM_C$.  In our situation the simple extension is $\oM\subset \oM(\fr) \to \sA$.  \Dan{tighten this}
\end{proof}

\subsection{Partial untwisting}

Let $\fr, \fr' : \sS \rightarrow \bZ_{>0}$ be two twisting choices.  We write that $\fr \prec \fr'$, and say that $\fr$ \emph{divides} $\fr'$, if $\fr(s)$ divides $\fr'(s)$ for every $s \in \sS$.  We saw in Definitions \ref{Def:r-twisted-T} and  \ref{Def:r-twisted-T} that $\cT^{\fr}$ can be constructed as a moduli space of enlargements of the log.\ structure on $\cT^{\sS}$ to one with characteristic monoid $M_{\cT}(\fr)$ and that its universal object $\sX^{\fr}$ is the moduli space of extensions of $M_{\sX}(\fr)$ admits a similar description.  If $\fr \prec \fr'$ then there are canonical inclusions $M_{\cT}(\fr) \subset M_{\cT}(\fr')$ and $M_{\sX}(\fr) \subset M_{\sX}(\fr')$ inducing a commutative diagram
\begin{equation*} \xymatrix{
  \sX^{\fr'} \ar[r] \ar[d] & \sX^{\fr} \ar[d] \\
  \cT^{\fr'} \ar[r] & \cT^{\fr} .
} \end{equation*}
Moreover, these maps are compatible with the projections to $\sX^{\sS}$ over $\cT^{\sS}$, so we obtain
\begin{proposition}
  If $\fr, \fr' : \sS \rightarrow \bZ_{>0}$ are twisting choices such that $\fr \prec \fr'$ then there is a map $\cT^{\fr'} \rightarrow \cT^{\fr}$, compatible with the projections to $\cT^{\sS}$, that is an isomorphism on dense open substacks of source and target.  If $\fr \prec \fr' \prec \fr''$, the resulting triangle is commutative in a canonical way.
\end{proposition}
\subsection{Split expansions}
\label{sec:split-expansions}

It is particularly important in degeneration formulas to understand the fiber of $\fT \rightarrow \sA$ over the origin $0\in \AA^1 \to\sA$, and the corresponding families of expanded degenerations.  Let $\cT_0 = 0 \fp_{\sA} \fT$ be this locus, and  let $\cT_0^{\twst} = 0 \fp_{\sA} \fT^{\twst}$ where $\fT^{\twst} \rightarrow \sA$ is the composition of the untwisting map $\fT^{\twst} \rightarrow \fT$ and the projection $\fT \rightarrow \sA$.

Define $\cT_{\spl} = \cT \times \cT$.  If $(\sY_1,\sE_1)$ and $(\sY_2, \sE_2)$ denote two expanded pairs over $S$ then we can attach $\sY_1$ to $\sY_2$ along the isomorphisms $\sE_1 \xrightarrow{\sim} \sD \xleftarrow{\sim} \sE_2$ (cf.\ \cite[Appendix~A]{AGV}).  This gives a map $\Tspl \rightarrow \cT_0$. 

We refer to \cite{Costello} for the notion of maps of pure degree.

\begin{proposition}
  $\Tspl \rightarrow \cT_0$ has pure degree~$1$.
\end{proposition}
\begin{proof}
The morphism is an isomorphism above the dense open substack  of $\cT_0$ where the fibers of the singular locus of the universal expansion have one irreducible component. \end{proof}

The proposition above is sufficient to understand splitting in the untwisted case.  The twisted case is less straightforward.  Instead of a direct map to relate $\cTtw_{\spl}$ to $\cTtw_0$, we have a correspondence
\begin{equation*}
  \xymatrix{
    & \cTtw_{\spl} \ar[dl] \ar[dr] \\
    \cT^{\rm tw} \fp_{\bZ_{>0}} \cT^{\rm tw} & & \cT_{\spl} \fp_{\cT_0} \cTtw_0 \ar[r]^<>(0.5){\pi} & \cTtw_0 .
  }
\end{equation*}
(All of the stacks in the diagram will be defined below.)  The map $\pi$ is obtained from base change from $\cT_{\spl} \rightarrow \cT_0$ so is of pure degree $1$.  We will see below that both of the other two maps are of the same pure degree $\frac{1}{\br}$, which permits one, using Costello's theorem \cite[Theorem~5.0.1]{Costello} or Manolache's refinement \cite[Proposition~2]{Manolache}, to translate directly between one stack and the other in questions of virtual enumerative geometry \cite[Section~5]{AF}.

Let $\Ttwspl$ be the stack whose $S$-points are pairs of twisted expanded pairs $(\sA'_1, \sD'_1)$ and $(\sD'_2, \sD'_2)$, together with a band-inverting isomorphism of $\Gm$-gerbes $\sD'_1 \rightarrow \sD'_2$ and a commutative diagram
\begin{gather*}
  \xymatrix{
    \sD'_1 \ar[rr] \ar[dr] & & \sD'_2 \ar[dl] \\
    & \sD
  }
\end{gather*}
where the maps $\sD'_i \rightarrow \sD$ are those induced by untwisting.  Note that there is no such isomorphism if the twisting of the distinguished divisors $\sD'_1$ and $\sD'_2$ are different.  Locally in $S$ such an isomorphism always exists and the $2$-automorphism group of such an isomorphism is canonically isomorphic to $\mu_r$.  This proves the first half of the following.

\begin{proposition}
  Let $\br$ be the locally constant function recording the order of twisting at the splitting divisor.
  \begin{enumerate}
  \item   The forgetful map $\cTtw_{\spl} \rightarrow \cT^{\rm tw} \fp_{\bZ_{>0}} \cT^{\rm tw}$ makes $\cTtw_{\spl}$ into a gerbe banded by $\mu_{\br}$ over $\cT^{\rm tw} \fp_{\bZ_{>0}} \cT^{\rm tw}$.
  \item   The map $\cTtw_{\spl} \rightarrow \cTtw_0 \fp_{\cT_0} \cT_{\spl}$ has pure degree $\frac{1}{\br}$.
  \end{enumerate}
\end{proposition}

For the second claim, it is enough to consider a dense open substack of $\cTtw_0 \fp_{\cT_0} \cT_{\spl}$.  Let $\cU \subset \cT$ be the open substack of $\fT$ over which the corresponding expansion of $\AA^2 \rightarrow \AA$ has at most one node in the fibers;  define $\cU^{\twst}$ to be the pullback of $\cU$ under the untwisting map, $\cU_0$ its intersection with $\cT_0$, etc.  Note that $\cU_{\spl} = \cU_0$, so $\cU^{\rm tw}_0 \fp_{\cU_0} \cU_{\spl} = \cU^{\rm tw}_0$.

Now, $\cU \cong \sA$ and $\cU^{\rm tw}$ is isomorphic to the union of infinitely many copies of $\sA$, indexed by $\bN$, and joined along their open points.  If the open substack labelled by $r$ is identified with $\sA$ then the untwisting map $\cU^{\rm tw} \rightarrow \cU$ is identified with the $r$-th power map.  Thus
\begin{gather*}
  \cU^{\rm tw}_0 \cong \coprod_{r \in \bN} (\sA, [r]) \fp_{\sA} 0  \cong \coprod_{r \in \bN} \Big[ \big( \Spec \bC[t] / t^r \big) \bigg/ \mu_r \Big]
\end{gather*}
The notation in the second term is meant to indicate the fiber product of $0 \rightarrow \sA$ and the $r$-th power map $[r] : \sA \rightarrow \sA$.  On the other hand,
\begin{gather*}
  \cU^{\rm tw}_{\spl} \cong \coprod_{r \in \bN} B \mu_r \cong \coprod_{r \in \bN} \Big[ \big( \Spec \bC[t] / t \big) \bigg/ \mu_r \Big]
\end{gather*}
That is, on the $r$-th component, $\cU^{\rm tw}_0$ is given by the equation $t^r = 0$ and $\cU^{\rm tw}_{\spl}$ is given by $t = 0$.  This proves the proposition. \qed

\section{Other incarnations of the stacks of expansions}\label{Sec:otherdef}

The stack $\cT$, its sister $\fT$ and their twisted cousins were described in terms of the stack $\sA$. Since the stack $\sA$ itself has modular interpretations, we can in turn interpret the various stacks of expansions in corresponding modular terms. 

In Section \ref{Sec:configs} we use the fact that $\sA$ parametrize line bundles with sections to interpret  $\cT$ et al.\ as a stacks of {\em configurations of line bundles}. The stack $\sA$ also serves as the stack of rank-1 Deligne--Faltings logarithmic structures, an open substack of Olsson's stack $LOG$ of all logarithmic structures. 

In Section \ref{Sec:aligned} this is used to interpret $\cT$ as a stack of {\em aligned logarithmic structures}. The connection is made through the work of Borne and Vistoli relating logarithmic structures and configurations of line bundles. 

Finally, the stacks of configurations of line bundles have a readily available \'etale covering by stacks of the form $\sA^n$, giving an immediate concrete construction of $\cT$ in terms of a categorical colimit of stacks, see section \ref{Sec:colimits}.

\subsection{Aligned logarithmic structures}\label{Sec:aligned}

It is convenient to start with the logarithmic picture.

\begin{definition}
An \emph{aligned logarithmic structure} on a scheme $S$ is a locally free logarithmic structure $M \rightarrow \cO_S$, together with a subset $\oM_0 \subset \oM$ such that on each fiber the generators of $\oM$ can be labelled $\set{e_1, \ldots, e_n}$ so that the corresponding fiber of $\oM_0$ consists of the elements
\begin{equation*}
0, e_1, e_1 + e_2, e_1 + e_2 + e_3, \ldots, e_1 + \cdots + e_n .
\end{equation*}
Let $\Tlog$ be the moduli stack of aligned logarithmic structures.

We similarly define a twisted version: let $\Tlog^\tw$ be the moduli stack of aligned logarithmic structures along with a simple extension $M\hookrightarrow N$.
\end{definition}

We construct an aligned logarithmic structure on $\Tpairs$.  Let $\sX \rightarrow \Tpairs$ be the universal expansion of the pair $(\PP^1, 0)$.  Then $\sX$ is a family of pre-stable rational curves, hence gives rise to a logarithmic structure $M$ on $\Tpairs$.  Let $\oM$ be the characteristic monoid of $M$.  The generators in a fiber of $\oM$ over $\Tpairs$ are in bijection with the nodes in the corresponding fiber of $\sX$.  In particular, they are totally ordered by their positions on the chain of rational curves.  

Let $\oM_0$ be the sheaf on $\Tpairs$ whose sections are subsets of $M$ satisfying the following two conditions:\Dan{I do not understand this}
\begin{enumerate}[label=(\roman{*})]
\item if $\sigma$ is a section of $\oM_0$ and $f \in \sigma$ then $f$ restricts to a generator in each fiber of $\oM$, and
\item if $f$ is an element of a section $\sigma$ of $\oM_0$ and $e$ is a section of $\oM$ that restricts to a generator in the fibers with $e < f$ in every fiber, then $e \in \sigma$.
\end{enumerate}
There is an embedding $\oM_0 \rightarrow \oM$ sending a subset $\sigma$ of $\oM$ to $\sum_{f \in \sigma} f$.  One can verify easily by looking at the fibers that this map is injective and $\oM_0 \subset \oM$ is an aligned logarithmic structure.

\begin{proposition}
The map $\cT \rightarrow \Tlog$ described by the construction above is an isomorphism. Similarly $\cT^\tw \rightarrow \Tlog^\tw\times \ZZ_{>0}$, where the second factor records the twisting along $D$, is an isomorphism.  
\end{proposition}

\begin{proof}[Sketch of proof] It is evident that the functor induces a bijection on isomorphism classes of geometric objects: for $\cT$ these are determined by the number of nodes and for $\Tlog$ by the cardinality of $\oM_0$.  Considering automorphisms, the functor in fact gives an equivalence on geometric points.  Indeed the automorphism group of an accordion of length $n$ is $\Gm^n$, and the same holds for a free log structure with fixed basis of the characteristic.  It is also easy to see that the infinitesimal deformation spaces agree so the morphism is \'etale, giving the result. The twisted case follows from Olsson's theory of log-twisted curves.
\end{proof}

One can similarly define $\fT_{\log}$ and $\fT_{\log}^\tw$ with isomorphisms $\fT \to \fT_{\log}$ and $\fT^\tw \to \fT_{\log}^\tw$.

\subsection{Configurations of line bundles}\label{Sec:configs}

It is possible to remove the sheaf of monoids in the description sketched above and retain only the sheaf  $\oM_0$.  We say that a sheaf $E$ on the \'etale site of a scheme is a \emph{sheaf of totally ordered finite sets}\Dan{Should this be in the previous section describing the nature of $\oM_o$?} if $E$ is a sheaf of partially ordered sets and
\begin{enumerate}[label=(\roman{*})]
\item $E$ is constructible, and
\item any two sections of $E$ are locally comparable.
\end{enumerate}
There is a stack associated to $E$ whose objects are the sections of $E$.  If $x$ and $y$ are two sections of $E$, there is a unique morphism from $x$ to $y$ if and only if $x \geq y$.

Let $\sL$ be the category of line bundles; its fibers $\sL(S)$ are denoted $\mathfrak{Pic}\: S$ in \cite{BV}.  An $S$-point of $\sL$ is a line bundle on $\sL$, and arrows are arbitrary arrows of line bundles, which may vanish.  Note that $\sL$  differs from $\BGm$ and is not a stack in groupoids: a morphism between line bundles is induced from a morphism of their underlying $\Gm$-torsors if and only if it is an isomorphism.

\begin{definition}
Let $\Ttoset$ be the stack whose objects are pairs $(E, L)$ where $E$ is a sheaf of totally ordered finite sets, and $L : E \rightarrow \sL$ is a morphism of stacks such that if $x \geq y$ are sections of $E$ and $L(x) \rightarrow L(y)$ is an isomorphism then $x = y$.  We refer to $\Ttoset$ as the stack of \emph{totally ordered configurations of line bundles}.
\end{definition}

Somewhat heuristically, we can describe an object of $\Ttoset$ as a diagram of line bundles $\cL_n \to \cL_{n-1}\to \cdots \to\cL_1 \to \cO$, with the agreement that, when $\cL_i \to \cL_{i-1}$ is an isomorphism, the diagram is equivalent to the one with $\cL_{i-1} $ removed and the arrow $\cL_i \to \cL_{i-1}$ is the composition of  $\cL_i \to \cL_i\to \cL_{i-1}$.\Jonathan{this isn't very accurate}

If $M$ is a logarithmic structure and $\oM_0$ is an alignment of $M$ then for each section $x$ of $\oM$, the fiber of $M$ over $x$ is a $\Gm$-torsor $P(x)$, equipped with a $\Gm$-equivariant morphism to $\cO = P(0)$.  For each $x \in M$, let $L(x)$ be the line bundle associated to the $\Gm$-torsor $P(x)$.  If $x \geq y$ then there is some $z \in \oM$ such that that $x = y + z$.  We get a map $L(z) \rightarrow \cO$ and this induces a map
\begin{equation*}
L(x) \simeq L(y) \tensor L(z) \rightarrow L(y) \tensor \cO \simeq L(y) .
\end{equation*}
These maps fit together to give a morphism of stacks $L : \oM_0 \rightarrow \sL$.  This construction defines a morphism $\Tlog \rightarrow \Ttoset$.

\begin{proposition}
The map $\Tlog \rightarrow \Ttoset$ described above is an equivalence.
\end{proposition}

This is an immediate application of Borne and Vstoli's interpretation of logarithmic structures in terms of line bundles, given in \cite{BV}.

\Dan{I want here a quick description of the universal family}

\Dan{What's the right thing for twisted??}


\subsection{Colimits}\label{Sec:colimits}

On $\sA^n$ there is a natural sequence of homomorphisms of line bundles.  Let $(L_1, s_1), \ldots, (L_n, s_n)$ be the ``coordinates'' on $\sA^n$, viewing $\sA$ as the moduli space of line bundles with sections.  Then we have a sequence of homomorphisms,
\begin{equation*}
L_1^\vee \tensor \cdots \tensor L_n^\vee \xrightarrow{s_n} L_1^\vee \tensor \cdots \tensor L_{n-1}^\vee \xrightarrow{s_{n-1}} \cdots \xrightarrow{s_3} L_1^\vee \tensor L_2^\vee \xrightarrow{s_2} L_1^\vee \xrightarrow{s_1} \cO .
\end{equation*}
Let $[n]$ denote the totally ordered set $\set{0 \leq 1 \leq \cdots \leq n}$ and let $[n]_{\sA^n}$ be the constant sheaf of totally ordered sets on $\sA^n$ associated to $[n]$.  The sequence above defines a map $L : [n]_{\sA^n} \rightarrow \sL$.  Define $E$ to be the quotient of $[n]_{\sA^n}$ by the relation $i \sim j$ if $L(i) \rightarrow L(j)$ is an isomorphism.  This gives a map $\sA^n \rightarrow \Ttoset$.

If $u : [n] \rightarrow [m]$ is a morphism of totally ordered sets it induces a morphism $\sA^u : \sA^n \rightarrow \sA^m$.  If $(L_i, s_i)$ are the coordinates on $\sA^m$ then ${\sA^u}^\ast (L_i, s_i) = \bigotimes_{u(j) = i} (L_j, s_j)$.  If $u$ is an injection then $\sA^u$ is an open embedding and it is not hard to check that the diagram 
\begin{equation*} \xymatrix{
\sA^n \ar[rr] \ar[dr] & & \sA^m \ar[dl] \\
& \Ttoset
} \end{equation*}
is commutative.

\begin{proposition}
The maps $\sA^n \rightarrow \Ttoset$ are \'etale and induce an equivalence
\begin{equation*}
\varinjlim_n \sA^n \xrightarrow{\sim} \Ttoset
\end{equation*}
where the colimit is taken in the $2$-category of stacks.
\end{proposition}
\begin{proof}
Since the $\sA^n$ are all representable and \'etale over $\Ttoset$, it is enough to see that the map induces an equivalence on geometric points, which is clear.
\end{proof}

\Dan{a couple of words about why this is true?}

\makeatletter
\@addtoreset{thm}{section}
\@removefromreset{thm}{subsection}
\counterwithin{thm}{section}
\makeatother
\appendix 


\section{Balanced $\Gm$ action on twisted rational curves}
\label{sec:Gm-action}

Let $C$ be a chain of twisted rational curves corresponding to a geometric point $s$ of $\fM_{0,2}^{\tw\,\sst}$, so each of the end components has a marked point, which itself may be twisted.  If $C$ is not twisted, there is a canonical balanced action of $\Gm$ on $C$ by scaling all components of the chain simultaneously.   We show below that there is also a balanced action on a chain of twisted curves, which uniquely extends to families.  In fact, there are two canonical actions, due to the inversion automorphism of $\Gm$, so we choose the convention that the first marked point corresponds to $0$ and the second marked point to $\infty$.

\subsection*{Action of a group on a stack} Defining an action of a group on a stack can be more subtle than an action on a scheme. We avoid such complications by building the quotient stack into the definition:

\begin{definition}
  Let $G$ be an algebraic group-scheme, $BG$ its classifying stack, and $E \to BG$ the universal torsor.  An action of $G$ on a stack $X$ is the data of  a stack $\oX$, a morphism  $\oX\to BG$, and an isomorphism $X \simeq \oX \fp_{BG} E$.
\end{definition}

\subsection*{Balanced action on fibers} We first define the balanced action on fibers.

\begin{proposition} \label{prop:tw-chain-Gm}
  Let $C$ be a chain of twisted rational curves over a field with a single, possibly twisted, marking on each end of the chain.  Then there is a canonical action of $\Gm$ on $C$ which induces the canonical balanced action on the coarse moduli space.
\end{proposition}
\begin{proof}
  We define the action of $\Gm$ on $X = [\AA^1 / \mu_r]$ and then glue.  Let $\oX = \sA$ where the map $\oX \rightarrow \BGm$ sends a line bundle and section $(L, s)$ to $L^{\tensor r}$.  Then $\oX \fp_{\BGm} E$ is the moduli space of $r$-th roots of a trivialized line bundle and section---the $r$-th root stack of $\AA^1$ at the origin---which is precisely $X$.

  We now obtain the action of $\Gm$ on twisted curves by gluing together the actions on root stacks of $\AA^1$.  Note that gluing along the open orbit of $\AA^1$ is trivial, since the quotient is a point. At a node, the quotient of
\begin{equation*}
  X = [\AA^1 / \mu_r ] \mathop{\sqcup}_{[0 / \mu_r]} [\AA^1 / \mu_r],
\end{equation*}
glued along the automorphism of $B\mu_r$ induced from the inversion map on $\mu_r$, is $\sX = \sA \sqcup_{\BGm} \sA$, glued along the \emph{inversion} isomorphism $\BGm \simeq \BGm$.  Indeed, $\sX$ may be identified as the moduli space of $(L,x,y)$ where $x$ is a section of $L$, the section $y$ is a section of $L^{\tensor (-1)}$ and $xy = 0$. The structure map $\sX \to \BGm$ sends $(L,x,y)$ to $L^{\otimes r}$.  Taking the fiber product with $E \rightarrow \BGm$ gives the same data, along with a trivialization of $L^{\tensor r}$, which is canonically isomorphic to $X$.
\end{proof}

\subsection*{Extension to families}

\begin{proposition} \label{prop:Gm-action}
  The balanced action defined above extends uniquely to every family of curves in $\fM_{0,2}^{\tw\,\sst}$.
\end{proposition}
\begin{proof}
{\sc Step 1: Infinitesimal extension.}  Suppose that $S$ is the spectrum of an artinian local ring, and $S'$ is a square-zero extension whose ideal is isomorphic to the residue field.  Let $C'$ be an $S'$-point of $\fM_{0,2}^{\tw\,\sst}$ and let $C$ be its fiber over $S$.  Assume that we have already constructed a $\Gm$ action on $C$ over $S$ whose restriction to the residue field is the balanced action described above, and that this extension is unique.  We show that it can also be extended to an action on $C'$ over $S'$.

Denote the projection by $\pi : C \rightarrow S$ and let $\pi_0 : C_0 \rightarrow S_0$ be the fiber of $\pi$ over the residue field.  The isomorphism classes of extensions of $C$ to $S'$ form a torsor under 
\begin{equation*}
  \Ext^1(\LL_{C/S}, \pi^\ast \cO_{S_0}) = \Ext^1(\LL_{C_0/S_0}, \cO_{C_0}) = \sum_x T_x C_1 \tensor T_x C_2
\end{equation*}
where the sum is taken over the nodes $x$ of $C_0$, and $C_1, C_2$  are the components of $C_0$ joined at $x$.

The isomorphism classes of equivariant extensions form a torsor under
\begin{equation*} \begin{split}
  \Ext^1(\LL_{[C/\Gm] / S \times \BGm}, \opi^\ast \cO_{S_0}) & = \Ext^1(\LL_{[C_0/\Gm] / S \times \BGm}, \cO_{\oC_0}) \\ & = \Ext^1(\LL_{C_0/S_0}, \cO_{C_0})^{\Gm} ,
\end{split}
\end{equation*}
the group of $\Gm$-invariants in $\Ext^1(\LL_{C_0/S_0}, \cO_{C_0})$.  But $\Gm$ acts trivially on $T_x C_1 \tensor T_x C_2$ since the action is balanced, so isomorphism classes of deformations and of equivariant deformations form pseudo-torsors under the same group. The same argument applied to $\Ext^2$ shows that equivariant obstructions coincide with obstructions, which vanish, so these pseudo-torsors are torsors. Therefore the map from equivariant deformations to non-equivariant deformations is a bijection on isomorphism classes.

To see that it is actually an equivalence, we must consider the automorphisms of an equivariant deformation, which can be identified with the $\Gm$-invariants in the group of all automorphisms:
\begin{equation*}
  \Hom(\Omega_{C_0/S_0}(p_1 + p_2), \cO_{C_0})^{\Gm} \subset \Hom(\Omega_{C_0/S_0}(p_1 + p_2), \cO_{C_0}) .
\end{equation*}
Since a vector field on a chain of rational curves must vanish at the nodes, an element of $\Hom(\Omega_{C_0/S_0}(p_1 + p_2), \cO_{C_0})$ corresponds to one vector field on each of the irreducible components of $C_0$, vanishing at the two special points of that component.  There is a one-dimensional space of vector fields on $\PP^1$ that vanish at $0$ and $\infty$. These vector fields lift uniquely to any twisting of $\PP^1$ and are equivariant with respect to the $\Gm$-action fixing $0$ and $\infty$.  Therefore the inclusion displayed above is an isomorphism, and every automorphism of $C'$ is equivariant.

We may deduce by induction that whenever $S$ is an infinitesimal thickening of a point and $C$ is an $S$-point of $\fM_{0,2}^{\tw\,\sst}$ there is a unique action of $\Gm$ on $C$ over $S$ extending the balanced action on the fiber.

{\sc Step 2: algebraization.}  Suppose that $C$ is an $S$-point of $\fM_{0,2}^{\tw\,\sst}$ where $S = \Spec R$ and $R$ is a complete noetherian local ring.  By the above steps, we have a balanced action of $\Gm$ on $C_n = C \fp_S S_n$ where $S_n = \Spec R / \fm^n$ and $\fm$ is the maximal ideal of $R$.  We wish to extend the action to $C$.

It is equivalent to show that for any pair of maps $T \rightarrow S$ and $t : T \rightarrow \Gm$ the formal automorphism of $C_T$ over $T$ given by $t$ can be algebraized uniquely.  Since we are proving that the algbraization is unique, we can work locally and assume $T$ is affine.  Since $\fM_{0,2}^{\tw\,\sst}$ is locally of finite presentation, we can also assume that $T$ is locally of finite presentation over $S$.  Then $T$ is noetherian and adic, so by \cite[Th\'eor\`eme~5.4.1]{ega-3-1} a formal morphism $C_T \rightarrow C_T$ can be algebraized uniquely.  This gives the algebraization.

{\sc Step 3: approximation.}  To construct the action on an arbitrary family $C \rightarrow S$, it is sufficient to construct it \'etale locally, and since $\fM_{0,2}^{\tw\,\sst}$ is locally of finite presentation, it is even enough to assume that $S$ is the spectrum of a henselian local ring $R$.  Using local finite presentation once again, we can even assume that $R$ is the henselization at a prime ideal of a ring that is of finite type over $\ZZ$.\Jonathan{removed reference to excellence}

Let $T \rightarrow S$ and $t : T \rightarrow \Gm$ be any maps, as before.  It is enough to show that we can construct the action of $t$ on $C_T$, and it is enough to do this locally.  Using the local finite presentation of $\fM_{0,2}^{\tw\,\sst}$, it is enough to assume that $T$ is the spectrum of the henselization of an algebra of finite type over $R$.  We can therefore apply Artin's approximation and find that the automorphism defined by $t$ over $\hT$ can be approximated by an automorphism over $T$.

We can therefore approximate the action of $\Gm$ on $C_S$ arbitrarily well in an \'etale neighborhood of each geometric point of $\Gm$.  Since $\Gm$ is quasi-compact, we can approximate the action of $\Gm$ on $C$ arbitrarily well on an \'etale cover of $\Gm$.  Moreover, we saw before that the infinitesimal deformation theory of this action is trivial so the approximation to the action over $C_{\hS}$ just constructed agrees to all orders with the action on $C_{\hS}$.  Therefore the approximation is unique and descends to $\Gm$ from the \'etale cover on which it was constructed above.  This completes the proof. 
\end{proof}

\subsection*{Functoriality}

\begin{proposition} \label{prop:equivariant}
  Any morphism of~$2$-marked twisted rational curves is $\Gm$-equivariant with respect to the canonical $\Gm$-action.
\end{proposition}
\begin{proof}
Suppose $f : C_1 \rightarrow C_2$ is a morphism of $2$-marked rational curves over $S$ and $\lambda$ is a point of $\Gm$.  The locus of $x \in C_1$ where $\lambda . f(x) = f(\lambda . x)$ is closed since $C_2$ is separated over $S$.  Its complement is therefore open in $C_1$ and has an open image in $S$ since $C_1$ is flat over $S$.  Therefore the locus in $S$ where $f \circ \lambda = \lambda \circ f$ is closed.  But this locus clearly contains all of the geometric points of $S$, so it must be all of $S$. 
\end{proof}

\bibliographystyle{amsalpha}
\bibliography{ACFW}

\end{document}